\documentclass[11pt,a4paper,eqno]{amsart}
\usepackage[left=0.9in, right=0.9in, top=1.5in, bottom=1.5in]{geometry}
\usepackage[utf8]{inputenc}
\usepackage[T1]{fontenc} 
\usepackage[english]{babel}
\usepackage{yhmath}
\usepackage{amsmath,amssymb}
\usepackage{stmaryrd}
\usepackage{amsmath}
\usepackage{amsthm}
\usepackage{tikz}
\usetikzlibrary{arrows}
\usetikzlibrary {patterns}
\usepackage{verbatim}

\usepackage{graphicx}
\usepackage{enumitem}
\usepackage{footnote} %
\usepackage{appendix}
\usepackage{color}
\usepackage[colorlinks=true,urlcolor=blue,citecolor=blue,linkcolor=blue,linktocpage,pdfpagelabels,bookmarksnumbered,bookmarksopen]{hyperref}
\allowdisplaybreaks

\pgfarrowsdeclare{<<<}{>>>}
{
\arrowsize=0.2pt
\advance\arrowsize by .5\pgflinewidth
\pgfarrowsleftextend{-4\arrowsize-.5\pgflinewidth}
\pgfarrowsrightextend{.5\pgflinewidth}
}
{
\arrowsize=0.2pt
\advance\arrowsize by .5\pgflinewidth
\pgfpathmoveto{\pgfpointorigin}
\pgfpathlineto{\pgfpoint{-1.2mm}{-.8mm}}
\pgfusepathqstroke
\pgfpathmoveto{\pgfpointorigin}
\pgfpathlineto{\pgfpoint{-1.2mm}{.8mm}}
\pgfusepathqstroke
}

\newtheorem{theorem}{Theorem}[section]

\newtheorem{corollary}[theorem]{Corollary}
\newtheorem{lemma}[theorem]{Lemma}

\newtheorem{proposition}[theorem]{Proposition}
\newtheorem{definition}[theorem]{Definition}
\newtheorem{remark}[theorem]{Remark}

\newtheorem*{theorem*}{Theorem}

\newcommand{\R}{\mathbb{R}}

\newcommand{\N}{\mathbb{N}}

\newcommand{\rmnote}[1]{}

\usepackage{braket}
\usepackage{esint}

\title[On the optimization of the Robin eigenvalues in some classes of polygons]{On the optimization of the Robin eigenvalues in some classes of polygons}

\author{Alessandro Carbotti$^1$}
\address{$^1$Dipartimento di Matematica e Fisica \lq\lq E. De Giorgi\lq\lq, Università del Salento, Via per Arnesano, 73100 Lecce, Italy.}
\email{alessandro.carbotti@unisalento.it}
\author{Simone Cito$^1$}
\address{$^1$Dipartimento di Matematica e Fisica \lq\lq E. De Giorgi\lq\lq, Università del Salento, Via per Arnesano, 73100 Lecce, Italy.}
\email{simone.cito@unisalento.it}
\author{Diego Pallara$^2$}
\address{$^2$Dipartimento di Matematica e Fisica \lq\lq E. De Giorgi\lq\lq, Università del Salento, Via per Arnesano, 73100 Lecce, Italy and INFN, Sezione di Lecce}
\email{diego.pallara@unisalento.it}



\begin{document}

\maketitle

\begin{center}
    {\em To the memory of Umberto Massari}
\end{center}
	\begin{abstract}
	Given the eigenvalue problem for the Laplacian with Robin boundary conditions, (with $\beta\in\R\setminus\{0\}$ the Robin parameter), we consider shape minimization problems if $\beta>0$ and shape maximization problems if $\beta<0$. Both problems are settled in a suitable class of generalized polygons with an upper bound on the number of sides, under either perimeter or volume constraint. 
	\end{abstract}
	
	\tableofcontents
	\section{Introduction}
    \label{sec:intro}
    
    The classical Robin eigenvalue problem is formulated as follows: given a nonzero real parameter $\beta$, we look for which values of $\lambda$ the boundary value problem 
    \begin{equation}\label{eq:classicalRobin}
        \left\{\begin{array}{ll}
        \Delta u + \lambda u = 0 \qquad &\text{in\ \ } \Omega
        \\
        \displaystyle \frac{\partial u}{\partial\nu} + \beta u =0 & \text{on\ \ }  \partial\Omega
        \end{array}\right.
    \end{equation}
    admits a nonzero weak solution. Here $\Omega$ is a bounded set in $\R^d$ with sufficiently smooth boundary and $\nu$ is the outer normal. It is well known (see e.g. \cite{BucFreKen}) that this problem admits a weak formulation based on the bilinear form
    \begin{equation}\label{eq:RobinForm}
        \mathcal{E}_\beta (u,v) = \int_\Omega \nabla u\cdot\nabla v\:dx+\beta\int_{\partial\Omega}uv\:d\mathcal{H}^{d-1},\quad u,v\in H^1(\Omega),
    \end{equation}
    where $\mathcal{H}^{d-1}$ is the $(d-1)$-dimensional Hausdorff measure and $u$ on $\partial\Omega$ is the boundary trace. By the trace inequality in the Sobolev space $H^1(\Omega)$ the bilinear form $\mathcal{E}_\beta$ is semibounded in $L^2(\Omega)$, i.e., there are constants $c_1,c_2>0$ such that $\mathcal{E}_\beta(u,u)+c_1\|u\|^2_{L^2}\geq \|u\|_{H^1}^2$, hence $\mathcal{E}_\beta$ defines a closed operator $(-\Delta_\beta)$ in $L^2(\Omega)$ which is self-adjoint and has compact resolvent, so that its spectrum is real and consists of an increasing sequence $\bigl(\lambda_{k,\beta}\bigr)_{k\in\N}$ of eigenvalues such that $\lambda_{k,\beta}\to+\infty$ as $k\to +\infty$, with $\lambda_{1,\beta}>0$ if $\beta>0$. The corresponding variational formulation for the first eigenvalue leads to the minimization of the Rayleigh quotient
\begin{equation}\label{RaileighQuotient}
R_{\Omega,\beta}(u):=
\frac{\displaystyle\int_\Omega|\nabla u|^2\:dx+\beta\int_{\partial\Omega}u^2\:d\mathcal{H}^{d-1}}{\displaystyle\int_\Omega u^2\:dx},
\end{equation}
\begin{equation}\label{eq:Rayleigh1}
\lambda_{1,\beta}(\Omega)=\min_{u\in H^1(\Omega)\setminus\{0\}}
R_{\Omega,\beta}(u).
\end{equation}
As in the more classical Dirichlet and Neumann problems (that can be formally obtained from \eqref{eq:classicalRobin} by letting $\beta$ to $+\infty$ and to $0$, respectively) one can look for the sets $\Omega$ that optimize suitable functions of the Robin spectrum. Such problems are different in nature according to the sign of the boundary parameter $\beta$. The simplest one is obviously that of finding the set $\Omega$ with minimal first eigenvalue with $\beta>0$ among the sets with prescribed Lebesgue measure: the optimal set turns out to be the ball, as proved in dimension two by M.-H. Bossel \cite{Bossel} and in higher dimension by D. Daners \cite{Daners}. In the negative boundary parameter regime $\beta<0$, by using a constant test function in \eqref{eq:Rayleigh1} it is easily checked that the first eigenvalue verifies 
\begin{equation}\label{negativebeta}
\lambda_{1,\beta}(\Omega)\leq \beta\mathcal{H}^{d-1}(\partial\Omega)/|\Omega|<0, 
\end{equation}
where $|\Omega|$ is the Lebesgue measure of $\Omega$. As a consequence, the first eigenvalue is bounded from above but one should expect it to be bounded below under either the perimeter constraint or the measure constraint. Therefore it is very natural to study the {\em maximization} of the first eigenvalue. In this case the optimal set is not yet known in general, though in the planar case with measure constraint it is proved in \cite{FreKre} that the disk is a maximizer for sufficiently small (in absolute value) values of the boundary parameter and that there exists a critical threshold under which the disk is not a maximizer anymore.  On the other hand, in \cite{bfnt19} it is proved that the ball is the only maximizer among convex sets, in any dimension and for any value of $\beta<0$.

Generalizing the Courant-Fischer formula \eqref{eq:Rayleigh1} to higher eigenvalues, we consider 
\begin{equation}\label{eq:Rayleighk}
\lambda_{k,\beta}(\Omega)=\min_{S\in\mathcal{S}_k}\max_{u\in S\setminus\left\{0\right\}} R_{\Omega,\beta}(u)
\end{equation}
where $\mathcal{S}_k$ denotes the set of all $k$-dimensional subspaces of $H^1(\Omega)$. Notice that the presence of boundary integrals makes the study of the Robin problem deeply different from the Dirichlet and Neumann problems. 

In this paper we study two shape optimization problems (according to the sign of $\beta$) for some functions of the Robin eigenvalues, in the planar case $d=2$, where the class of competitors is that of polygons. In order to tackle these optimization problems, we fix a suitable topology on the polygons (the one induced by the $H^c$ convergence, see Definition \ref{def:HcConvergence}), we enlarge the class of simple polygons in Definition \ref{def:poly} to encompass their $H^c$ limits, i.e., the {\em generalized polygons} introduced in Definition \ref{def:genpol}, and we extend the variational characterization of the eigenvalues to such class, see \eqref{eq:genEigenvalues}. 
    
Though the formulation of spectral shape optimization problems in some subclasses of polygons is rather simple, there are several issues in this setting. Let us start with the \textit{polygonal Faber-Krahn} inequality stated by Polya and Sz\"ego, see \cite[Pag. 158]{polyaszego}: among polygons with at most $N$ sides and given area, does the regular $N$-agon minimize the first Dirichlet eigenvalue? Even if the result seems natural and expected, a direct proof of it is available only for $N=3$ and $N=4$. Indeed, in these two cases, the classic symmetrization techniques work and thus it is natural to obtain the equilateral triangle and the square as minimal polygons, respectively. 
On the other hand, if $N\ge 5$, the Steiner symmetrization of a polygon could increase the number of sides, in general. In  \cite[Section 3.3]{henrot06} it is proved that among polygons with \emph{at most} $N$ sides, optimizers exist and have \emph{exactly} $N$ sides. This result is proved by showing that a small cut near a convex corner produces a better competitor with more sides. This idea is exploited in Section \ref{sec:positive} below. Anyway, the question of the precise shape of the optimal polygons remains open. Recently, it has been pointed out in \cite{BogBuc, bogbucvalidated} that for some $N\ge 5$ the proof that the optimal set is the regular $N$-agon can be reduced to a finite number of certified numerical computations and the local minimality of the regular pentagon and hexagon has been shown. Concerning other boundary conditions, the possibility to handle explicit eigenfunctions or to separate the variables plays an important role. For what concerns Neumann conditions, it is worth mentioning  \cite{vdbbucgit}, where the authors address the problem of maximizing the Neumann eigenvalues on rectangles with a measure or perimeter constraint. Instead, concerning Robin eigenvalues with positive boundary parameter $\beta>0$, in \cite{freken} the authors have proved that the square minimizes the first eigenvalue among all (unions of) rectangles of a given area. For the higher eigenvalues, they proved that the square (respectively,  the union of $k$ equal squares) minimizes $\lambda_{1,\beta}$ (respectively, $\lambda_{k,\beta}$) among rectangles (respectively, unions of rectangles) of given area if $\beta$ is below a certain threshold; on the other hand, they showed that the optimizers are not the square or the union of $k$ equal squares if $\beta$ is large enough. It is worth mentioning also \cite{laug19}, where the author showed how the rectangular case supports some well-known conjectures about spectral shape optimization problems involving the Robin eigenvalues. For the case $\beta<0$, we mention \cite{krlotu}, where the authors proved that the equilateral triangle locally maximizes the first eigenvalue among all triangles of a given area, again provided that $|\beta|$ is below a certain threshold depending only upon the area constraint. We also cite \cite{FreLau}, where the maximization of the second Robin eigenvalue in general space dimension for general sets is addressed, and it is shown that in a suitable range of values of the boundary parameter that includes those for which the second eigenvalue remains positive, under volume constraint, the maximizing set is the ball. 

When dealing with the negative boundary parameter case, in some situations it turns out that the perimeter constraint is rather natural, see for instance  \cite{bfnt19,cito2021quantitative}, where the optimality and the stability of the ball for the first eigenvalue in the convex case is addressed, \cite{AntFreKre} where the maximality of the ball is proved among all bounded planar sets of class $C^2$ with fixed perimeter for all the values of the (negative) boundary parameter, or also \cite{Vikulova} for some extensions in dimension $d=3$. Even in our framework, the (generalized) perimeter constraint is very helpful to obtain some additional properties of the optimal shapes, see Proposition \ref{Prop:furtherproperties}. 

As highlighted in the previous round-up of references, for the polygonal case we have very little information about the optimizers even for the first eigenvalue (except for special cases in which either the eigenfunctions are explicit or the symmetrization techniques work, see \cite{McCartin, freken}, or some restrictions on the admissible polygons). The main difficulty in tackling this problem is that it is not possible to transpose the same argument used in \cite{Bossel, Daners} to prove an isoperimetric inequality for the first Robin eigenvalue. Indeed, such results are based on the radiality of the first eigenfunction of the disk and on a comparison between the (smooth) level sets of such function and the level sets of an eigenfunction of a generic domain. On the other hand, the possible presence of parts of the boundary with multiplicity 2 is a challenging problem even in more general settings, see \cite{citogiaco}. Moreover, in the Robin case, the optimality of the regular $N$-agon is ensured only if $N\in\{3,4\}$ within some ranges of the boundary parameter. For this reason, our intent is to get more general existence results for both cases $\beta>0$ and $\beta<0$, without any assumption on the magnitude of the parameter, the number of sides and the order of the eigenvalue. More precisely, we focus on a wider class of spectral functionals, whose prototype is the sum of the first $k$ eigenvalues, we consider both the perimeter and measure constraint and we prove the existence of solutions and some qualitative properties, combining well established techniques holding in more general settings (generalization of the eigenvalues, existence in a weaker framework, continuity of the traces along moving boundaries, etc.) and peculiar features of the polygonal case.

The paper is organized as follows: in Section \ref{sec:Prelimin} we describe our framework, recall the necessary preliminary results and state  our main results. Sections \ref{sec:positive} and \ref{sec:negative} are devoted to the proofs of our main results, namely Theorems \ref{mtheorem:main1} and \ref{Th:main2}, respectively. In Section \ref{sec:Open} we discuss some further issues.

\noindent
\paragraph*{\bf Acknowledgements}  
We are grateful to Prof. Pedro Freitas for his interest in the paper and his suggestions. We also thank the anonymous referee for many valuable suggestions and comments that helped us to improve the first version of the paper.

The authors are members of GNAMPA of the Istituto Nazionale di Alta Matematica (INdAM). 
A.C. acknowledges the support of the INdAM - GNAMPA 2025 Project ``Metodi variazionali per problemi dipendenti da operatori frazionari isotropi e anisotropi'' and of the INdAM - GNAMPA 2026 Project 
	``Analisi variazionale per operatori locali e nonlocali possibilmente singolari o degeneri''.
S.C. acknowledges the support of the INdAM - GNAMPA 2025 Project ``Disuguaglianze funzionali di tipo geometrico e spettrale''.
The authors have been also partially supported by the PRIN 2022 project 20223L2NWK.

\noindent
{\bf Data Availability Statement} Data sharing is not applicable to this article as no datasets were generated or analysed during the current study.

\section{Preliminaries and Main results}\label{sec:Prelimin}

In order to introduce the class of admissible polygons, we start from the definition of simple polygon. In the whole paper by {\em line segment} we always mean a {\em maximal} subset of a straight line in the plane belonging to the boundary of a polygon.

Following \cite{buf}, we introduce the class of admissible polygonal sets. 

\begin{definition}[Simple polygons]\label{def:poly}
A simple polygon is the open bounded planar region $P$ delimited by a finite number of not self-intersecting line segments (called sides) which are pairwise joined at their endpoints (called vertices) to form a simple closed path.
\end{definition}

Let us denote by $\mathcal{P}_{N}$ the family of simple polygons with at most $N$ sides. Notice that simple polygons are connected and simply connected.

In the following, we use as a key tool the $H^c$-convergence, as it preserves many topological properties of polygonal domains. Let us start from the Hausdorff distance in $\R^2$.

\begin{definition}[Hausdorff distance]\label{def:H-distance}
   Let $A,B \subset \R^2$ be closed. We define the Hausdorff distance between $A$ and $B$ by
\[
d_H (A, B) := \max \left\{\sup_{x\in A} {\rm dist}(x, B), \sup_{x\in B} {\rm dist}(x, A) \right\}.
\] 
\end{definition}

The Hausdorff convergence of a sequence of open sets is defined using the distance between their complements. 

\begin{definition}[Hausdorff convergence of open sets]\label{def:HcConvergence}
  Let $D \subset \R^2$ be compact and let $E, F \subset \R^2$ be open. We define the Hausdorff complementary distance in $D$ between $E$ and $F$ by
\[
d_{H^c} (E, F) = d_H (D \setminus E, D \setminus F).  
\]
If $E_n\subset D$ for every $n\in\N$, we say that $E_n$ $H^c$-converges to $E\subset D$ if $\lim_n d_{H^c}(E_n,E)=0$ in $D$. For a general sequence $E_n\subset\R^2$, we say that $E_n$ locally $H^c$-converges to $E$ if for any ball $B$ the sequence $E_n\cap B$ $H^c$-converges to $E\cap B$.
\end{definition}
It is easily seen that the definition of $H^c$-convergence is independent of the choice of the compact set $D$. 

Notice that in general the $H^c$-limit of a sequence of simple polygons in $\mathcal{P}_{N}$ is not a simple polygon in $\mathcal{P}_{N}$, as shown in Figure \ref{fig:polypacman}.

\begin{figure}[h!]
	\centering
		\includegraphics[width=0.9\textwidth]{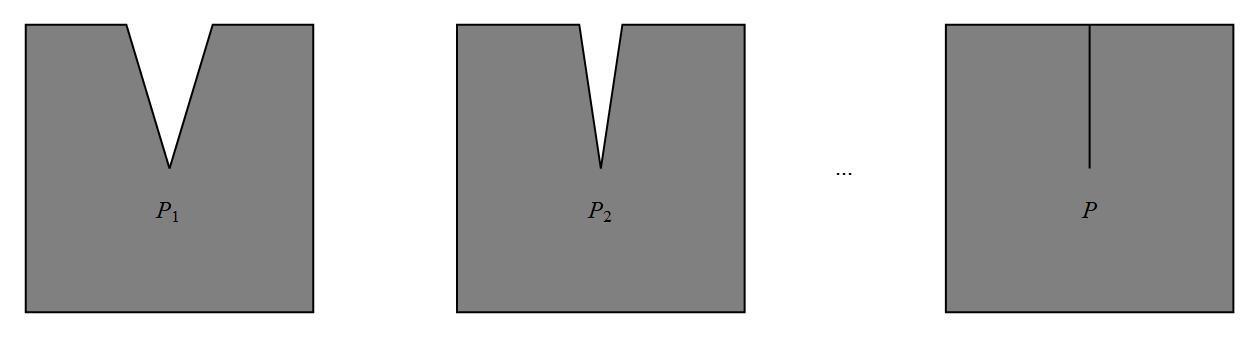}
	\caption{The sequence $(P_n)\subset\mathcal{P}_7$ $H^c$-converges to the ``degenerate polygon'' $P$, which has a boundary given by line segments, but is not a simple polygon.}
	\label{fig:polypacman}
\end{figure}

To overcome this problem, we follow the approach in \cite{buf} and set our shape optimization problems in a class of sets that contains the $H^c$-limits of simple polygons.

\begin{definition}[Generalized polygons]\label{def:genpol}
We say that an open set $P\subset\R^2$ is a generalized polygon with at most $N$ sides if there exists a sequence $(P_n)$ of simple polygons in $\mathcal{P}_N$ such that $(P_n)$ locally $H^c$-converges to $P$ and $\displaystyle\limsup_{n\to\infty}|P_n|<\infty$. We denote by $\overline{\mathcal{P}_{N}}$ the class of generalized polygons with at most $N$ sides.
\end{definition}

\begin{remark}[counting the sides of a generalized polygon]
The number of sides of a generalized polygon $P$ is the least number of line segments such that $\partial P$ is a (connected) closed path. We say that a side  has multiplicity 2 if it appears twice in such curve. For instance, the polygon $P$ in Figure \ref{fig:polypacman} has boundary given by the union of 5 segments, but according to our convention it has 7 sides.
\end{remark}

Notice that if $P$ is a generalized polygon, any side having double multiplicity has at least one vertex on the topological boundary of $\overline{P}$.

\medskip

The following compactness result is contained in \cite[Proposition 4.6.1]{bucur2004variational}

\begin{proposition}\label{Prop:compactness}
Let $D\subset\mathbb{R}^d$ a fixed compact set. Then, the class of the open sets contained in $D$ is compact in the Hausdorff-complementary topology.
\end{proposition}

\begin{remark}[see Remark 2.2.18 in \cite{hpi}]\label{rem:connectionopen}
If $\Omega_n\stackrel{H^c}{\longrightarrow}\Omega$ into a fixed compact $B$, then, denoting by $\# E$ the number of connected components of $E$, $\#(\Omega^c\cap B)\le\liminf_n \#(\Omega^c_n\cap B)$.

In dimension $d=2$ this allows us to obtain further topological information: if a bounded open set $\Omega\subset\R^2$ is a disjoint union of simply connected open sets, then, for any compact set $B\subset\R^2$, $B\setminus\Omega$ is a compact connected set. This implies that the $H^c$-limit of unions of simply connected set is union of simply connected set. Indeed, let $(\Omega_n)$ be a sequence of open bounded subsets of $\R^2$ such that each $\Omega_n$ is a bounded disjoint union of simply connected open sets; if $\Omega_n\stackrel{H^c}{\longrightarrow}\Omega$, then
$$
1\le\#(\Omega^c\cap B)\le\liminf_n \#(\Omega^c_n\cap B)=1
$$
and so $\Omega$ is union of simply connected open sets.
\end{remark}

\begin{remark}\label{pro:compoly}
The following facts hold true for the family $\overline{\mathcal{P}_{N}}$ (see \cite{buf} for details).
\begin{itemize}
\item[(i)] $\overline{\mathcal{P}_N}$ is closed with respect to the local $H^c$-convergence since the number of connected components of the complement of each generalized polygon is uniformly bounded.
\item[(ii)] Every $P\in\overline{\mathcal{P}_N}$ is union of simply connected generalized polygons, since $P^c$ is connected.
\item[(iii)] $P\in\overline{\mathcal{P}_N}$ may be disconnected; each connected component of $P$ is delimited by a finite number of line segments (still called the sides of $P$), which are pairwise joined at their endpoints (still called vertices of $P$) to form a closed path, possibly containing self-intersections; in particular, $P$ has at most $N$ sides, counted with their multiplicity. $P$ has at most $\lfloor\frac{N-1}{2}\rfloor$ connected components. If $N$ is odd, this upper bound is obtained for instance if $P$ is the union of triangles with consecutive bases lying on the same line (so their union is considered as one side, according to our definition of line segment). If $N$ is even, the upper bound is obtained by replacing one of the triangles in the previous construction with a quadrilateral.

\begin{figure}[h!]
	\centering
		\includegraphics[width=0.9\textwidth]{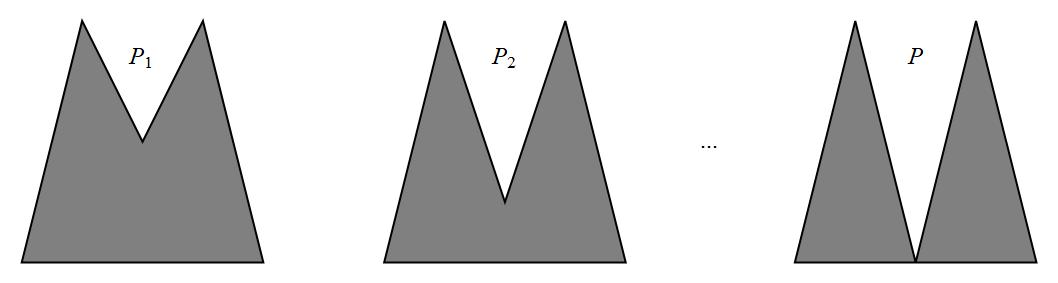}
	\caption{The sequence $(P_n)\subset\mathcal{P}_5$ $H^c$-converges to $P\in\overline{\mathcal{P}_5}$, that has $2=\lfloor\frac{5-1}{2}\rfloor$ connected components.}
	\label{fig:mountains}
\end{figure}

\item[(iv)] Since the topological boundary of the closure of any $P\in\overline{\mathcal{P}_N}$ is a closed curve, $P$ is bounded and thus has finite Lebesgue measure. 
Conversely, since any $P\in\overline{\mathcal{P}_N}$ has finite Lebesgue measure, it is bounded (otherwise, in view of the bound on the number of sides, necessarily $P$ would have two parallel sides with infinite length, contradicting the fact that $|P|<+\infty$.)
\end{itemize}
\end{remark}

\begin{remark}
Let us observe that the number of sides is lower semicontinuous for locally $H^c$-converging sequences $(P_n)\subset\overline{\mathcal{P}_{N}}$. 

Indeed, let $(P_n)$ be a sequence of generalized polygons $H^c$ converging to $P\in\overline{\mathcal{P}_{N}}$ and let $M\leq N$ be the biggest integer such that there are infinitely many $P_n$ in $\overline{\mathcal{P}_{M}}$. Then, by Remark \ref{pro:compoly}(i) the limit $P$ belongs to $\overline{\mathcal{P}_{M}}\subseteq\overline{\mathcal{P}_{N}}$. 
Notice that this fact does not hold if the number of sides is not bounded a priori (the sequence $(R_n)$ of regular $n$-agons of measure $m$ centered at a point $x_0\in\R^2$ $H^c$-converges to the disk of measure $m$ centered at $x_0$).
\end{remark}

The definition of the trace of $u\in H^1(P)$, where $P$ is a simple polygon, is defined in the usual way. Following \cite[Section 1.1.7]{Grisvard}, we can define the trace of $u$ also when $P$ is a generalized polygon, that may lie on both sides of an inner boundary segment. To this aim, we can assume that there is only one inner boundary segment $S$ (otherwise, divide $P$ in smaller polygons with such a property). Then adding a further segment $S'$ starting from the inner endpoint of $S$, we divide $P$ as the union of two simple polygons and in each of them the trace of $u$ is well defined. Of course, the traces of $u$ on both sides of $S'$ coincide, whereas on $S$ they can be different. Henceforth, we denote by $u^+$ and $u^-$ these traces. In the sequel, it is not important to make explicit a criterion to distinguish the two sides, which we call {\em right} and {\em left} only to simplify the presentation. Moreover, to deal with simple and generalized polygons at the same time, we agree that all functions defined in $P$ are extended as 0 out of $\overline{P}$ and denote by $u^+$ the interior trace on $\partial \overline{P}$.

In order to take into account inner boundary segments, we relax the definition of the Rayleigh quotients \eqref{eq:Rayleighk} and we define the {\em generalized Rayleigh quotient} by
\begin{equation}\label{GenRaiQuotient}
\overline{R}_{P,\beta}(u):=
    \frac{\displaystyle\int_P|\nabla u|^2\:dx+\beta\int_{\partial P}\left[(u^+)^2+(u^-)^2\right]\:d\mathcal{H}^{1}}{\displaystyle\int_P u^2\:dx}
\end{equation}
and the {\em generalized eigenvalues} on a generalized polygon $P$ by
\begin{equation}\label{eq:genEigenvalues}
\overline{\lambda}_{k,\beta}(P):=\inf_{S\in\mathcal{S}_k}\max_{u\in S\setminus\left\{0\right\}}\overline{R}_{P,\beta}(u),
\end{equation}
where as above $\mathcal{S}_k$ denotes the set of all $k$-dimensional subspaces of $H^1(P)$. 
This definition is well posed, since it does not depend on the orientation of $\partial P$. Moreover, if $P$ is a simple polygon, then $\overline{\lambda}_{k,\beta}(P)=\lambda_{k,\beta}(P)$, as on the boundary $u^+=u$ and $u^-=0$.

\bigskip

We recall two useful properties of the classical eigenvalues $\lambda_{k,\beta}$ (see \cite{BucFreKen}) that are generalized to $\overline{\lambda}_{k,\beta}$ in a standard way.

\begin{remark}[\bf Some properties of generalized eigenvalues] Let $P\in\overline{\mathcal{P}_N}$, $|P|<+\infty$. The following properties hold:
\label{rem:scaling}
\begin{itemize}
\item $\beta\mapsto\overline{\lambda}_{k,\beta}(P)$ is strictly increasing
\item For $t>0$ we have
$$
\overline{\lambda}_{k,\beta}(t P)=\frac{1}{t^2}\overline{\lambda}_{k,t\beta}(P).
$$
In particular, when $\beta>0$, we have 
\begin{equation}\label{eq:scal}
\overline\lambda_{k,\beta}(t P)<\frac{1}{t}\overline{\lambda}_{k,\beta}(P)<\overline{\lambda}_{k,\beta}(P)
\end{equation}
for every $t>1$.
\item If $P=P_1\cup P_2$, with $P_1,P_2\in\overline{\mathcal{P}_N}$ nonempty and disjoint, it holds
\begin{equation}\label{eq:compconn}
    \overline{\lambda}_{h,\beta}(P) =\min_{
i=0,...,h}
\max\left\{\overline{\lambda}_{i,\beta}(P_{1}),\overline{\lambda}_{h-i,\beta}(P_{2})\right\}
\end{equation}
(where we set $\overline{\lambda}_{0,\beta}(P_i):=0$).
\end{itemize}
\end{remark}

In particular, inequality \eqref{eq:scal} entails the decreasing monotonicity under dilation of the generalized eigenvalues with positive boundary parameter.

\begin{remark}
Let $P\in\mathcal{P}_N$ a simple polygon. A useful result (holding in general for every bounded, Lipschitz, connected, domain) is the following: the map $\beta\mapsto\lambda_{1,\beta}(P)$ is differentiable in $\R$ and it holds
    \begin{equation}\label{eq:4.16}
 \lim_{\beta\to 0}\left(\frac{d}{d\beta}  \lambda_{1,\beta}(P) \right)=\frac{\mathcal{H}^1(\partial P)}{|P|},
\end{equation}
see \cite[Formula (4.16)]{BucFreKen}.
More generally, if $P$ is a simple polygon whose $k$-th Neumann eigenvalue $\mu_k(P)$ is simple, then $\beta\mapsto\lambda_{k,\beta}(P)$ is differentiable in $\beta=0$ and it holds
    \begin{equation}\label{eq:4.12}
 \lim_{\beta\to 0}\left(\frac{d}{d\beta}  \lambda_{k,\beta}(P) \right)=\frac{\int_{\partial P}\psi_k^2\:d\sigma}{\int_P\psi_k^2\:dx},
\end{equation}
where $\psi_k$ is the corresponding Neumann eigenfunction, see \cite[Formula (4.12)]{BucFreKen}. An example of simple polygon having some simple higher Neumann eigenvalues is every non equilateral triangle $T$, whose second eigenvalue is always simple, see \cite[Theorem 1]{Siudeja}. A consequence of \eqref{eq:4.12} is that, expanding $\lambda_{2,\beta}(T)$ near $\beta=0$ and recalling $\lambda_{2,0}(T)=\mu_2(T)$, we get
\begin{equation}\label{eq:expansion}
\lambda_{2,\beta}(T)\simeq \mu_2(T)+\beta\frac{\int_{\partial T}\psi_2^2\:d\sigma}{\int_T\psi_2^2\:dx}
\end{equation}
as $\beta\to 0$.
\end{remark}

\bigskip

As we are dealing with convergent sequences of (generalized) polygons, the corresponding Rayleigh quotients are settled on different function spaces whose convergence in turn must be defined. The right notion of convergence is the {\em Mosco convergence} (see e.g. \cite[Chapter 4]{bucur2004variational}), as it combines a geometric notion of convergence of the domains with a functional-analytic convergence of the Sobolev spaces. 

\begin{definition}[Convergence in the sense of Mosco] \label{def:convMosco} 
Let $X$ be a Banach space and let $(G_n)$ be a sequence of closed subsets of $X$. We define weak upper and strong lower limits in the sense of Kuratowski the spaces
\[
w{\rm -}\limsup_{n\to+\infty} G_n =\left\{u\in X:\exists (n_k)_k,\ \exists u_{n_k} \in G_{n_k} \ {\rm s.t.\ }u_{n_k} \to u \ {\rm weakly\ in\ }X\right\},
\]
\[
s{\rm -}\liminf_{n\to+\infty}G_n =\left\{u\in X:\ \exists u_n \in G_n \ \text{s.t.\ }u_n\to u\ {\rm strongly\ in\ }X\right\}. 
\]
We say that $(G_n)$ Mosco-converges to $G$ if
\[
G= w{\rm -}\limsup_{n\to+\infty}G_n = s{\rm -}\liminf_{n\to+\infty} G_n.
\]
\end{definition}

Under suitable topological constraints, Mosco convergence of function spaces is equivalent to convergence in measure and Hausdorff convergence of the domains. An important result is the following theorem, see \cite[Theorem 7.2.1]{bucur2004variational}. 

\begin{theorem}\label{prop:Mosco-domains}
Let us denote by $\# E$ the number of connected components of the open set $E\subset\R^2$.
Let $\ell\in\N$ and let $(\Omega_n)$ be a sequence of open domains in $\R^2$  such that $(\Omega_n)$ is $H^c$-convergent to some $\Omega$, with $\#(\R^2\setminus\Omega_n) \leq \ell$ for every $n\in\N$. Then $H^1(\Omega_n)$ converges to $H^1(\Omega)$ in the sense of Mosco  if and only if $|\Omega_n|$ converges to $|\Omega|$. 
\end{theorem}

\medskip

In order to handle the possible fractures of a generalized polygon, we define the generalized perimeter of a polygon, recalling the approach in \cite{bugitr22,citogiaco}. Roughly speaking, we quantify the boundary length with the natural multiplicity.

\begin{definition}[Generalized perimeter]
Let $N\in\N$, $P\in\overline{\mathcal{P}_N}$. We set
$$
\partial^\star P:=\partial\overline{P}, \qquad
\Gamma:=\partial P\setminus\partial^\star P
$$
and we call generalized perimeter of $P$ the quantity
$$
\widetilde{Per}(P):=\mathcal{H}^1(\partial^*P)+2\mathcal{H}^1(\Gamma).
$$
\end{definition}

\begin{remark}\label{rem:nonconservper}
    $\widetilde{Per}(\cdot)$, in general, is only lower semicontinuous, see Figure \ref{fig:WW}.
\end{remark}
    \begin{figure}[h!]\label{fig:WW}
\begin{tikzpicture}[>=>>>]
\fill[red!80!green!100!] (-2,0) .. controls (-2,0) and (2,0) ..(2,0) .. controls (2,0) and (2,4) ..(2,4) .. controls (2,4) and (0.6,4) .. (0.6,4) .. controls (0.6,4) and (0.3,2) .. (0.3,2) .. controls (0.3,2) and (0,4) .. (0,4) .. controls (0,4) and (-0.3,2) .. (-0.3,2) .. controls  (-0.3,2)  and (-0.6,4) .. (-0.6,4) .. controls (-0.6,4) and (-2,4) .. (-2,4).. controls (-2,4) and (-2,0) ..  (-2,0);
\draw[line width=.7pt](-2,0) .. controls (-2,0) and (2,0) ..(2,0) .. controls (2,0) and (2,4) ..(2,4) .. controls (2,4) and (0.6,4) .. (0.6,4) .. controls (0.6,4) and (0.3,2) .. (0.3,2) .. controls (0.3,2) and (0,4) .. (0,4) .. controls (0,4) and (-0.3,2) .. (-0.3,2) .. controls  (-0.3,2)  and (-0.6,4) .. (-0.6,4) .. controls (-0.6,4) and (-2,4) .. (-2,4).. controls (-2,4) and (-2,0) ..  (-2,0);
\draw (0,4.5) node {$P_1$};
\fill[red!80!green!100!] (-2+6,0) .. controls (-2+6,0) and (2+6,0) ..(2+6,0) .. controls (2+6,0) and (2+6,4) ..(2+6,4) .. controls (2+6,4) and (0.2+6,4) .. (0.2+6,4) .. controls (0.2+6,4) and (0.1+6,2) .. (0.1+6,2) .. controls (0.1+6,2) and (0+6,4) .. (0+6,4) .. controls (0+6,4) and (-0.1+6,2) .. (-0.1+6,2) .. controls  (-0.1+6,2)  and (-0.2+6,4) .. (-0.2+6,4) .. controls (-0.2+6,4) and (-2+6,4) .. (-2+6,4).. controls (-2+6,4) and (-2+6,0) ..  (-2+6,0);
\draw[line width=.7pt](-2+6,0) .. controls (-2+6,0) and (2+6,0) ..(2+6,0) .. controls (2+6,0) and (2+6,4) ..(2+6,4) .. controls (2+6,4) and (0.2+6,4) .. (0.2+6,4) .. controls (0.2+6,4) and (0.1+6,2) .. (0.1+6,2) .. controls (0.1+6,2) and (0+6,4) .. (0+6,4) .. controls (0+6,4) and (-0.1+6,2) .. (-0.1+6,2) .. controls  (-0.1+6,2)  and (-0.2+6,4) .. (-0.2+6,4) .. controls (-0.2+6,4) and (-2+6,4) .. (-2+6,4).. controls (-2+6,4) and (-2+6,0) ..  (-2+6,0);
\draw (0+6,4.5) node {$P_2$};
\draw (9.5,2) node {\ldots};
\fill[red!80!green!100!] (11,0) .. controls (11,0) and (11,4) ..(11,4) .. controls (11,4) and (15,4) .. (15,4) .. controls (15,4) and (15,0) .. (15,0) .. controls (15,0) and (11,0) .. (11,0);
\draw[line width=.7pt](11,0) .. controls (11,0) and (11,4) ..(11,4) .. controls (11,4) and (13,4) .. (13,4) .. controls (13,4) and (13,2) .. (13,2) .. controls (13,2) and (13,4) .. (13,4) .. controls (13,4) and (15,4) .. (15,4).. controls (15,4) and (15,0) .. (15,0) .. controls (15,0) and (11,0) .. (11,0);
\draw (13,4.5) node {$P$};
\end{tikzpicture}
\caption{$\widetilde{Per}(P_n)\to 12 >10=\widetilde{Per}(P).$}
\end{figure}

Let us recall a condition that entails the compactness of the minimizing sequences. It is a reformulation of \cite[Lemma 4]{BucGia10} adapted to our framework.

\begin{lemma}\label{lem:conc}
Let $M>0$, $P\in\overline{\mathcal{P}_N}$ with $|P|<+\infty$ and $u\in H^1(P)$ such that $\|u\|_{L^2(P)}=L>0$ and assume that 
$$
\int_P|\nabla u|^2\:dx+\int_{\partial P}\left[(u^+)^2+(u^-)^2\right]\:d\mathcal{H}^1\le M.
$$
Then, there exist $y\in\R^2$ and a positive constant $C=C(|P|,M,L)$ such that
$$
|{\rm supp}(u)\cap Q_1(y)|\ge C(|P|,M,L),
$$
where $Q_1(y)$ is the square with center $y$ and sidelength 1. 
\end{lemma}

The lower semicontinuity of the boundary integral is proved in \cite[Lemma 19]{bucurcito}.

\begin{lemma}\label{teo:lscboundary}
Let $\Omega\subseteq\R^2$ be open, $k\in\N$, and let $(K_n),K\subset\Omega$ be compact sets with at most $k$ connected components, such that $\limsup_n\mathcal{H}^1(K_n)<+\infty$ and $K_n\to K$ in the Hausdorff metric.  Let $u_n\in H^1(\Omega\setminus K_n)$ be such that
\begin{equation}\label{eq:lscboundary29}
\limsup_{n\to+\infty}\|u_n\|_{H^1(\Omega\setminus K_n)} +\int_{K_n}\left[(u_n^+)^2+(u_n^-)^2\right]\:d\mathcal{H}^1<+\infty.
\end{equation}
Then, there exists $u\in H^1(\Omega\setminus K)$ such that, up to subsequences, we have
$$
u_n\to u\quad\text{strongly in $L^2_{loc}(\Omega)$},
$$
$$
\nabla u_n\rightharpoonup \nabla u\quad\text{weakly in $L^2(\Omega;\R^2)$},
$$
and
\begin{equation}\label{eq:lscboundary}
\int_K\left[(u^+)^2+(u^-)^2\right]\:d\mathcal{H}^1\le\liminf_{n\to+\infty}\int_{K_n}\left[(u_n^+)^2+(u_n^-)^2\right]\:d\mathcal{H}^1.
\end{equation}
\end{lemma}

Notice that inequality \eqref{eq:lscboundary} holds even under the less restrictive hypothesis that $K$ is only a subset of the Hausdorff limit of $K_n$. This is the situation when, for instance, the open sets $\Omega_n$ $H^c$-converge to the open set $\Omega$: in that case, the compact sets $\partial\Omega_n$ $H$-converge to a compact set $K$, but, in general, $K$ is only a subset of $\partial\Omega$. See \cite[Proposition 2.2.16]{hpi} and the following discussion.

\bigskip

\paragraph{\textbf{Main results}}
Since the kind of optimization depends on the sign of the boundary parameter $\beta$, we split the discussion into two parts. Let $N\ge 3$ and $m,p>0$.

\medskip

Let us fix $\beta>0$ and let us study the problems
\begin{equation}\label{eq:poly1}
     \min\left\{\overline{\lambda}_{1,\beta}(P) :P\in\overline{\mathcal{P}_{N}},\ |P|\le m\right\},
    \end{equation}
\begin{equation}\label{eq:poly}
\min\left\{F(\overline{\lambda}_{1,\beta}(P),\ldots,\overline{\lambda}_{k,\beta}(P)):P\in\overline{\mathcal{P}_{N}},\ |P|\le m,\ P\subseteq D\right\},
\end{equation}
and
\begin{equation}\label{eq:polyper}
\min\left\{F(\overline{\lambda}_{1,\beta}(P),\ldots,\overline{\lambda}_{k,\beta}(P)):P\in\overline{\mathcal{P}_{N}},\ \widetilde{Per}(P)\le p\right\},
\end{equation}
where $D\subset\R^2$ is a nonempty compact set and $F\in C^1(\R^k)$ is nondecreasing in each variable with strictly positive derivative with respect to the first variable and such that
\begin{equation}\label{eq:fscoppia}
    \lim_{|\xi|\to+\infty}F(\xi)=+\infty.
\end{equation}
Our main result for $\beta>0$ is the following.

\begin{theorem}\label{mtheorem:main1}
    Problems \eqref{eq:poly1}, \eqref{eq:poly} and \eqref{eq:polyper} admit a solution in the class of generalized polygons. Any minimizer $P$ of \eqref{eq:poly1} or \eqref{eq:poly} verifies $|P|=m$ and any minimizer $P$ of \eqref{eq:polyper} verifies $\widetilde{Per}(P)=p$. Moreover, any solution of each problem has exactly $N$ sides counted with their multiplicity.
\end{theorem}

\medskip

    Let us now consider the case of a negative boundary parameter. To simplify the notation, we set $\eta:=-\beta>0$. As will be clarified in Remark \ref{rem:constraint}, the case of negative boundary parameter is more delicate since it is rather sensitive in terms of dilations or contractions of the domains. In other words, a bound from above on volume or on generalized perimeter allows sequences to shrink and, differently from the case of positive boundary parameter, each eigenvalue of order $k\ge 2$ can go both to $+\infty$ and $-\infty$. In view of the continuity of the volume under the convergence in measure of sets, it is then natural to consider the following maximization problem with volume constraint given in terms of \emph{prescription}:
\begin{equation}\label{eq:polyneg}
\max\left\{F(\overline{\lambda}_{1,-\eta}(P),\ldots,\overline{\lambda}_{k,-\eta}(P)):P\in\overline{\mathcal{P}_{N}},\ |P|= m\right\},
\end{equation}
where $F:\R^k\to\R$ is nondecreasing and upper semicontinuous in each variable and such that, for any variable $\xi_j$, it holds
$$
\lim_{\xi_j\to-\infty}F(\xi_1,\ldots,\xi_k)=-\infty.
$$
On the other hand, $\widetilde{Per}(\cdot)$ is not continuous, in general (see Remark \ref{rem:nonconservper}), unless we add some very restrictive hypotheses like convexity, but in this case $\widetilde{Per}(P)=\mathcal{H}^1(\partial P)$. For this reason, we consider the constraint "$\widetilde{Per}(P)\le p$" and we restrict our analysis to the maximization of the first eigenvalue.
We thus study the following problem
\begin{equation}\label{eq:polynegper}
\max\left\{\overline{\lambda}_{1,-\eta}(P): P\in\overline{\mathcal{P}_{N}},\ \widetilde{Per}(P)\le p\right\}.  
\end{equation}
Our main result for $\beta<0$ is the following.

\begin{theorem}\label{Th:main2}
    Problems \eqref{eq:polyneg} and \eqref{eq:polynegper} admit a solution in the class of generalized polygons. Every optimal polygon $P$ in \eqref{eq:polynegper} is union of simple polygons.
\end{theorem}

\section{Positive boundary parameter}\label{sec:positive}

We start this section with an important remark. 

 \begin{remark}[Generalized eigenvalues of a cracked set are actually eigenvalues]
 Following the approach of \cite{bugitr22} (see also \cite{citogiaco} for analogous results in any dimension), we get that the $\overline{\lambda}_{k,\beta}(P)$ are actually eigenvalues of the elliptic operator defined by the bilinear form
 $$
 \overline{\mathcal{E}}_\beta(u,v):=\int_P\nabla u\cdot\nabla v\:dx+\beta\int_{\partial P}\left[u^+v^++u^-v^-\right]\:d\mathcal{H}^{1},\quad u,v\in H^1(P).
 $$
 Clearly, the latter definition coincides with \eqref{eq:RobinForm} when $P$ is a simple polygon. The minimum in the Courant-Fischer formula is thus attained, namely
\begin{equation}\label{eq:genEigenvalues2}
\overline{\lambda}_{k,\beta}(P)=\min_{S\in\mathcal{S}_k}\max_{u\in S\setminus\left\{0\right\}}\overline{R}_{P,\beta}(u),
\end{equation}
where $\overline{R}_{P,\beta}(u)$ is defined in \eqref{GenRaiQuotient}. This fact allows us to test the Rayleigh quotient with actual eigenfunctions.
\end{remark}

\begin{remark}\label{rem:primapositivagen}
If $P$ is a connected generalized polygon and $u\in H^1(P)$ is a positive eigenfunction for $\overline{\lambda}_{1,\beta}(P)$ for some $\beta>0$, then there exists $\alpha>0$ such that $u\ge\alpha$, see \cite[Theorem 1]{BGN22}.
\end{remark}

An important result is the following adaptation to our context of \cite[Theorem 3.15]{citogiaco}.

\begin{theorem}[\bf Boundedness of the eigenfunctions]
\label{th:boundeigen}
Let $P\in\overline{\mathcal{P}_N}$. Then, there exists a positive constant $C$, depending only on $|P|$ and $\beta$, such that for every $L^2$-normalized eigenfuction $u\in H^1(P)$ for $\overline{\lambda}_{h,\beta}(P)>0$ it holds
\begin{equation}
\label{eq:uinfty}
\|u\|_\infty\le C \overline{\lambda}_{h,\beta}(P)^2.
\end{equation}
\end{theorem}

A consequence of the previous result is that, given a sequence of generalized polygons $P_n$ such that $|P_n|=m$ and $\sup\overline{\lambda}_{h,\beta}(P_n)<+\infty$, their $L^2$-normalized eigenfuctions $u_n\in H^1(P_n)$ are uniformly bounded in $L^\infty$.

\bigskip

In order to count the sides of the optimal polygons, we adapt a cutting technique used in \cite{citoconvex}. There, the author proved that optimal convex shapes for a class of Robin spectral functionals are $C^1$; the technique exploits an argument by contradiction addressed to remove possible corners. In the present framework, we cut a generalized polygon near the vertex of a convex corner to show that increasing the number of sides (within the prescribed constraint) decreases the value of the functional. Aiming at cutting a convex corner of the  polygon $P$, if there is such a corner with both sides of multiplicity one, we argue on that corner. If all the convex corners have a side of multiplicity two, we cut the corner only on one side as follows.

\begin{figure}[h!]
\begin{tikzpicture}[>=>>>]
\label{fig:cutpol}
\fill[orange!80!gray!] (-2,0) .. controls (-2,0) and (2,0.5) ..(2,0.5) .. controls (2,0.5) and (4,0.5) ..(4,0.5) .. controls (4,0.5) and (5,-.3) .. (5,-.3) .. controls (5,-.3) and (10,2.2) .. (10,2.2) .. controls (10,2.2) and (3.5,6) .. (3.5,6).. controls (3.5,6) and (0,5) .. (0,5).. controls (0,5) and (-2,0) .. (-2,0);
\draw[line width=.7pt]  (-2,0) .. controls (-2,0) and (2,0.5) ..(2,0.5) .. controls (2,0.5) and (4,0.5) ..(4,0.5) .. controls (4,0.5) and (5,-.3) .. (5,-.3) .. controls (5,-.3) and (10,2.2) .. (10,2.2) .. controls (10,2.2) and (3.5,6) .. (3.5,6).. controls (3.5,6) and (0,5) .. (0,5).. controls (0,5) and (-2,0) .. (-2,0);
\draw[line width=.7pt]  (-2,0) .. controls (-2,0) and (1,2) ..(1,2) .. controls (1,2) and (3,1.5) ..(3,1.5);
\draw[line width=.7pt] (5,-.3) .. controls (5,-.3) and (4.8,1.2) .. (4.8,1.2);
\draw[line width=.7pt] (7.2,2).. controls (7.2,2) and (10,2.2) .. (10,2.2);
\draw[line width=.7pt] (3.5,6) .. controls (3.5,6) and (1,4) .. (1,4) .. controls (1,4) and (1.1,3.5) .. (1.1,3.5);
\draw[line width=.7pt] (0,5).. controls (0,5) and (0.2,4.3) .. (0.2,4.3);
\draw (5,3.3) node[anchor=north east] {$P$};
\draw (3.5,6) node[anchor=south] {$x_0$};
\draw (3.6,6-3*.25) node {$\alpha_+$};
\draw (2.1,6-3*.25) node {$\alpha_-$};
\draw (2,4.3) node {$\ell_{x_0}$};
\end{tikzpicture}
	\caption{In the polygon $P$ all convex corners are determined by self-intersected sides; choosing the vertex $x_0$ as above, without loss of generality, we can argue only on $P\cap\alpha_+$ or $P\cap\alpha_-$.}
    \label{fig:cuttriangle}
	\end{figure}
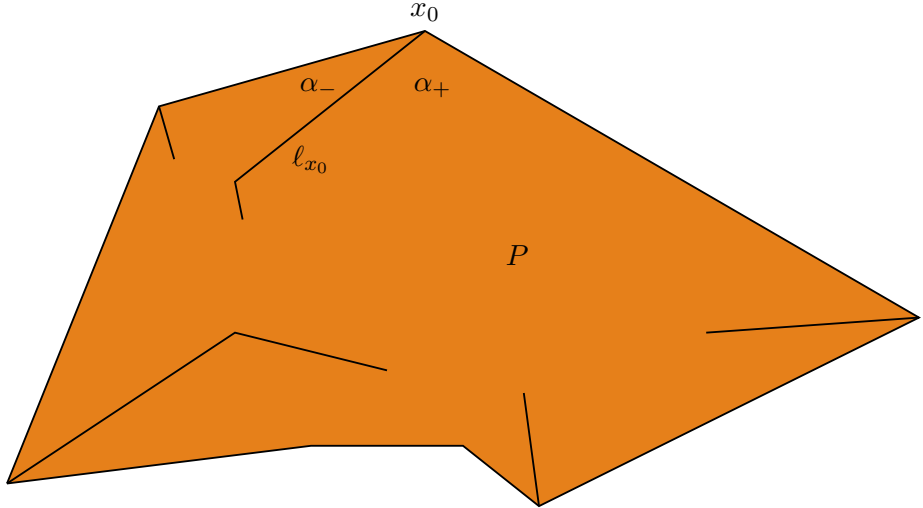

For definiteness, assume that $x_0=0$ is the vertex of a convex corner of the polygon $P$ and that the segment $\ell_{x_0}$ vertexed at $x_0$ has double multiplicity. Assume also that $P$ lies in the half-plane $\{x_2<0\}$ and apply the technique only on one side of the segment, say on $$P\cap\alpha_+=P\cap\{x\cdot\nu_{\ell_{x_0}}>0\},$$
where $\nu_{\ell_{x_0}}$ is one of the normal unit vectors to $\ell_{x_0}$. Accordingly, we define the following sets:
\begin{equation}\label{eq:2depsilonsets}
m_\varepsilon:=P\cap\alpha_+\cap\left\{x_2>-\varepsilon\right\},\ P_\varepsilon:=P\setminus m_\varepsilon,\ \sigma_\varepsilon:=P\cap\alpha_+\cap\left\{x_2=-\varepsilon\right\},\ s_\varepsilon:=\partial m_\varepsilon\setminus \sigma_\varepsilon.
\end{equation}
Notice that 
$$
\max_{x\in \overline{m_\varepsilon}}\text{dist}(x,\sigma_\varepsilon)=\text{dist}(0,\sigma_\varepsilon)=\varepsilon.
$$
\begin{figure}[h!]
\begin{tikzpicture}[>=>>>]
\label{fig:cutpoly}
\fill[orange!80!gray!] (-2,0) .. controls (-2,0) and (2,0.5) ..(2,0.5) .. controls (2,0.5) and (4,0.5) ..(4,0.5) .. controls (4,0.5) and (5,-.3) .. (5,-.3) .. controls (5,-.3) and (10,2.2) .. (10,2.2) .. controls (10,2.2) and (3.5+6.5*.4,6-3.8*.4) .. (3.5+6.5*.4,6-3.8*.4) .. controls (3.5+6.5*.4,6-3.8*.4) and (1.6,6-3.8*.4) .. (1.6,6-3.8*.4) .. controls  (1.6,6-3.8*.4)  and (3.5,6) .. (3.5,6) .. controls (3.5,6) and (0,5) .. (0,5).. controls (0,5) and (-2,0) .. (-2,0);
\draw[line width=.7pt]  (-2,0) .. controls (-2,0) and (2,0.5) ..(2,0.5) .. controls (2,0.5) and (4,0.5) ..(4,0.5) .. controls (4,0.5) and (5,-.3) .. (5,-.3) .. controls (5,-.3) and (10,2.2) .. (10,2.2) .. controls (10,2.2) and (3.5,6) .. (3.5,6).. controls (3.5,6) and (0,5) .. (0,5).. controls (0,5) and (-2,0) .. (-2,0);
\draw[line width=.7pt]  (-2,0) .. controls (-2,0) and (1,2) ..(1,2) .. controls (1,2) and (3,1.5) ..(3,1.5);
\draw[line width=.7pt] (5,-.3) .. controls (5,-.3) and (4.8,1.2) .. (4.8,1.2);
\draw[line width=.7pt] (7.2,2).. controls (7.2,2) and (10,2.2) .. (10,2.2);
\draw[line width=.7pt] (3.5,6) .. controls (3.5,6) and (1,4) .. (1,4) .. controls (1,4) and (1.1,3.5) .. (1.1,3.5);
\draw[line width=.7pt] (0,5).. controls (0,5) and (0.2,4.3) .. (0.2,4.3);
\draw[line width=.7pt] (3.5+6.5*.4,6-3.8*.4) .. controls (3.5+6.5*.4,6-3.8*.4) and (1.6,6-3.8*.4) ..(1.6,6-3.8*.4);
\draw[line width=.7pt,dashed] (1.6,6-3.8*.55) .. controls (1.6,6-3.8*.55) and (1.6,6-3.8*.4) ..(1.6,6-3.8*.4);
\draw[line width=.7pt,dashed] (3.5+6.5*.4,6-3.8*.4) .. controls (3.5+6.5*.4,6-3.8*.4) and (3.5+6.5*.4,6-3.8*.55) ..(3.5+6.5*.4,6-3.8*.55);
\draw (5,2.3) node[anchor=north east] {$P_\varepsilon$};
\draw (3.6,6-3.8*.25) node {$m_\varepsilon$};
\draw (3.5/2+6.5*.2+1.6/2,6-3.8*.55) node[anchor=south] {$\sigma_\varepsilon$};
\draw (4+6.5*.2,6.5-3.8*.2) node[anchor=south west] {$s_\varepsilon$};
\draw[<->,dashed] (3.5+6.5*.4,6-3.8*.55) -- (1.6,6-3.8*.55);
\draw[->,dashed] (4.1+6.5*.2,6.6-3.8*.2) -- (3.5+6.5*.2,6-3.8*.2);
\draw[->,dashed] (4.1+6.5*.2,6.6-3.8*.2) -- (2.55,6-3.8*.2);
\end{tikzpicture}
\caption{The cutting procedure of $P$}
\end{figure}
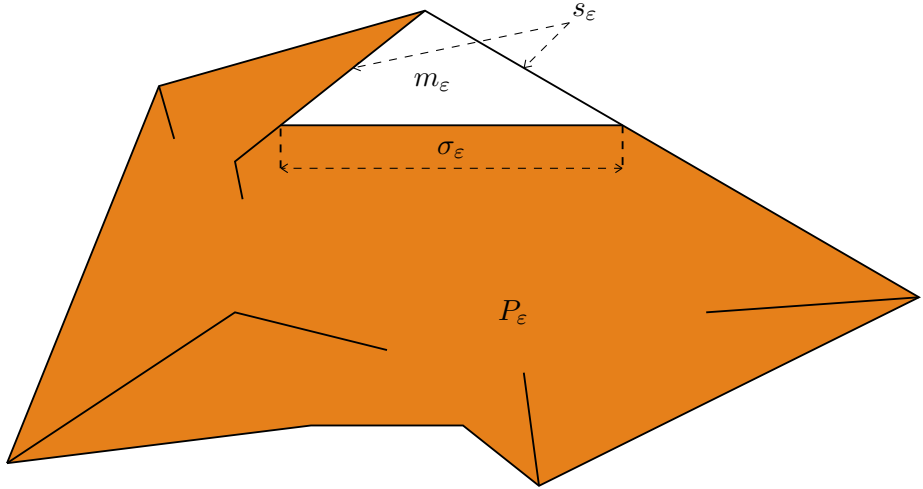

\begin{lemma}\label{lem:2dlambda1}
Let $P\in\overline{\mathcal{P}_N}$. Then, there exist $\varepsilon_0>0$ and $C=C(P,\beta)>0$ such that for every $0<\varepsilon<\varepsilon_0$, we have
\begin{equation}\label{eq:2dlambda1}
\overline{\lambda}_{1,\beta}(P_\varepsilon)\le \overline{\lambda}_{1,\beta}(P)-C\varepsilon.
\end{equation}
\begin{proof}
Let us assume that $P$ is connected. We start observing that both $\mathcal{H}^1(\sigma_\varepsilon)$ and $\mathcal{H}^1(s_\varepsilon)$ are infinitesimal of order $\varepsilon$, hence  $|m_\varepsilon|= \varepsilon\mathcal{H}^1(\sigma_\varepsilon)/2$ is infinitesimal of order $\varepsilon^2$ as $\varepsilon\to 0$. Moreover, since the $m_\varepsilon$ are triangles and the $\sigma_\varepsilon$ are parallel, there exists a constant $C_1>1$, depending only on $P$, such that
\begin{equation}\label{eq:trieps}
\mathcal{H}^1(s_\varepsilon)=C_1\mathcal{H}^1(\sigma_\varepsilon).
\end{equation}
Let us compare $\overline{\lambda}_{1,\beta}(P_\varepsilon)$ with $\overline{\lambda}_{1,\beta}(P)$. Let us consider $u\in H^1(P)$ an $L^2(P)$-normalized eigenfunction for $\overline{\lambda}_{1,\beta}(P)$ positively bounded away from zero (see Remark \ref{rem:primapositivagen}) and denote by $u_\varepsilon$ its restriction to $P_\varepsilon$ extended by zero outside $P_\varepsilon$. We have that $u_\varepsilon\in H^1(P_\varepsilon)$ is a test function for $\overline{\lambda}_{1,\beta}(P_\varepsilon)$ and it holds
$$(u_\varepsilon)^+=u,\ (u_\varepsilon)^-=0\quad \text{on $\sigma_\varepsilon$}$$
$$(u_\varepsilon)^+=u^+,\ (u_\varepsilon)^-=0\quad \text{on $\partial P_\varepsilon
\setminus\sigma_\varepsilon$}.
$$

\begin{figure}[h!]
\label{fig:zoomcut}
\begin{tikzpicture}[>=>>>, scale=2]
\fill[orange!80!gray!] (138/19,3.8) .. controls (138/19,3.8) and (3.5+6.5*.4,6-3.8*.4) .. (3.5+6.5*.4,6-3.8*.4) .. controls (3.5+6.5*.4,6-3.8*.4) and (1.6,6-3.8*.4) .. (1.6,6-3.8*.4) .. controls  (1.6,6-3.8*.4)  and (3.5,6) .. (3.5,6) .. controls (3.5,6) and (0,5) .. (0,5) .. controls (0,5) and (-0.48,3.8) .. (-0.48,3.8);
\draw[line width=.7pt] (138/19,3.8) .. controls (138/19,3.8) and (3.5,6) .. (3.5,6).. controls (3.5,6) and (0,5) .. (0,5) .. controls (0,5) and (-0.48,3.8) .. (-0.48,3.8);
\draw[line width=.7pt] (3.5,6) .. controls (3.5,6) and (1,4) .. (1,4) .. controls (1,4) and (1.04,3.8) .. (1.04,3.8);
\draw[line width=.7pt] (0,5).. controls (0,5) and (0.2,4.3) .. (0.2,4.3);
\draw[line width=.7pt] (3.5+6.5*.4,6-3.8*.4) .. controls (3.5+6.5*.4,6-3.8*.4) and (1.6,6-3.8*.4) ..(1.6,6-3.8*.4);
\draw (3.6,6-3.8*.25) node {$m_\varepsilon$};
\draw (2.75,6-3.8*.32) node {$(u_\varepsilon)^-=0$};
\draw (1.7,6-3.8*.22) node {$(u_\varepsilon)^+=u^+$};
\draw (3.5/2+6.5*.2+1.6/2,6-3.8*.55) node[anchor=south] {$(u_\varepsilon)^+=u$};
\draw[->,dashed] (1.75,6-3.8*.25) -- (2,4.8);
\draw[->,dashed] (2.9,6-3.8*.3) -- (2.55,6-3.8*.2);
\draw[->,dashed] (2.95,6-3.8*.34) -- (3,6-3.8*.4);
\draw[->,dashed] (3.5/2+6.5*.2+1.6/2+.2,6-3.8*.55+.3) -- (4.3,6-3.8*.4);
\end{tikzpicture}
\end{figure}

We thus have
\begin{equation}\label{eq:2dlambda1primoconto}
\begin{split}
\overline{\lambda}_{1,\beta}(P_\varepsilon)& \le \frac{\displaystyle\int_{P_\varepsilon}|\nabla u_\varepsilon|^2\:dx+\beta\int_{\partial P_\varepsilon}\left((u_\varepsilon^+)^2+(u_\varepsilon^-)^2\right)\:d\mathcal{H}^1}{\displaystyle\int_{P_\varepsilon} u_\varepsilon^2\:dx}
\\ 
&\le\frac{\displaystyle\int_{P}|\nabla u|^2\:dx+\beta\int_{\partial P}\left((u^+)^2+(u^-)^2\right)\:d\mathcal{H}^1+\beta\int_{\sigma_\varepsilon}u^2\:d\mathcal{H}^1-\beta\int_{s_\varepsilon}(u^-)^2\:d\mathcal{H}^1}{\displaystyle 1-\int_{m_\varepsilon} u^2\:dx}\\
&\le\left[\overline{\lambda}_{1,\beta}(P)+\beta\int_{\sigma_\varepsilon}u^2\:d\mathcal{H}^1-\beta\int_{s_\varepsilon}(u^-)^2\:d\mathcal{H}^1\right]\left(1+2\int_{m_\varepsilon} u^2\:dx\right)\\
&\le\overline{\lambda}_{1,\beta}(P)+ \beta\left(\int_{\sigma_\varepsilon}u^2\:d\mathcal{H}^1-\int_{s_\varepsilon}(u^-)^2\:d\mathcal{H}^1\right)+C_2\varepsilon^2
\end{split}
\end{equation}
for $\varepsilon$ small enough. We can consider the restriction of $u$ on $m_\varepsilon$ continuous on $\overline{m_\varepsilon}$, since it solves the mixed boundary value problem
$$
 \left\{\begin{array}{ll}
        \Delta v + \overline{\lambda}_{1,\beta}(P) u = 0 \qquad &\text{in\ \ } m_\varepsilon
        \\
        \displaystyle \frac{\partial v}{\partial\nu} + \beta v =0 & \text{on\ \ }  s_\varepsilon
        \\
        \displaystyle  v=u & \text{on\ \ }  \sigma_\varepsilon
        \end{array}\right.
$$
(see \cite{Mghazli}), so we can argue by continuity as in the Lipschitz case. To ease the readability, we denote such solution on $m_\varepsilon$ still by $u$. Let us fix $\delta>0$, with $(u(0)+\delta)^2\le C_1(u(0)-\delta)^2$. There exists $\varepsilon_0>0$ such that
$$
0<u(0)-\delta< u(x)<u(0)+\delta
$$
for every $x\in\overline{m_{\varepsilon_0}}$, since  $u(0)>0$.  In particular, the trace $u^-(x)$ of $u$ for $x\in s_\varepsilon$ satisfies the above estimate as well.

Now, $m_\varepsilon$ is decreasing in $\varepsilon$ with respect to inclusions, then we can choose $\varepsilon_0$ small enough so that $P_\varepsilon$ satisfies  \eqref{eq:2dlambda1primoconto}. Combining the latter with \eqref{eq:trieps} we get
\begin{equation*}
\begin{split}
\overline{\lambda}_{1,\beta}(P_\varepsilon)&\le\overline{\lambda}_{1,\beta}(P)+\beta\left(\mathcal{H}^1(\sigma_\varepsilon)(u(0)+\delta)^2-\mathcal{H}^1(s_\varepsilon)(u(0)-\delta)^2\right)+C_2\varepsilon^2\\
&\le\overline{\lambda}_{1,\beta}(P)+\beta\mathcal{H}^1(\sigma_\varepsilon)\left((u(0)+\delta)^2-C_1(u(0)-\delta)^2\right)+C_2\varepsilon^2\\
&=\overline{\lambda}_{1,\beta}(P)-\beta C_3\varepsilon+C_2\varepsilon^2\le\overline{\lambda}_{1,\beta}(P)-C\varepsilon,
\end{split}
\end{equation*}
where the last constant $C$ takes into account all the previous constants and depends only on the domain $P$ and on $\beta$.

The previous argument holds also if $P$ is not connected: it is enough to apply it on a connected component of $P$ realizing $\overline{\lambda}_{1,\beta}(P)$.
\end{proof}
\end{lemma}

Higher order eigenvalues are considered in the following result.

\begin{lemma}\label{lem:2dlambdak}
Let $P\in\overline{\mathcal{P}_N}$ be connected. Then, for every $h\in\N, h\ge 2$, 
\begin{equation}\label{eq:2dlambdak}
\overline{\lambda}_{h,\beta}(P_\varepsilon)\le \overline{\lambda}_{h,\beta}(P)+o(\varepsilon).
\end{equation}
\begin{proof}
Let $\left\{u_1,\ldots,u_h\right\}\subset H^1(P)$ be a $L^2$-orthonormal system of $h$ eigenfunctions associated with $\overline{\lambda}_{1,\beta}(P),\ldots,\overline{\lambda}_{h,\beta}(P)$ respectively and define $S:=\text{span}\left\{u_1,\ldots,u_h\right\}$. Then, for $\varepsilon$ sufficiently small,  $S_\varepsilon:=\text{span}\left\{u_1|_{P_\varepsilon},\ldots,u_h|_{P_\varepsilon}\right\}$ is a test space for the variational characterization \eqref{eq:genEigenvalues} of $\overline{\lambda}_{h,\beta}(P_\varepsilon)$. In particular, we can consider on $P_\varepsilon$ test functions of the form $\sum_{i=1}^h\alpha_i^\varepsilon u_i$ with $\sum_{i=1}^h\left(\alpha_i^\varepsilon\right)^2=1$. This in turn implies that, up to subsequences, $\alpha_i^\varepsilon\to\alpha_i\in[-1,1]$ and
$$
\sum_{i=1}^h\alpha_i^\varepsilon u_i\ \longrightarrow\ \sum_{i=1}^h\alpha_i u_i
$$
strongly in $H^1(P)$. Let us denote by $(\overline{\alpha}_1^\varepsilon,\ldots,\overline{\alpha}_h^\varepsilon)$ the $h$-tuple of coefficients that maximizes $\overline{R}_{P_\varepsilon,\beta}$ in $S_\varepsilon$ with $\sum_{i=1}^h\left(\overline{\alpha}_i^\varepsilon\right)^2=1$ and by $(\overline{\alpha}_1,\ldots,\overline{\alpha}_h)$ the $h$-tuple of coefficients such that, for every $i=1,\ldots,h$, $\overline{\alpha}_i^\varepsilon\to\overline{\alpha}_i$ (up to subsequences).
For any $\varepsilon>0$ sufficiently small, we can estimate $\overline{\lambda}_{h,\beta}(P_\varepsilon)$ using $S_\varepsilon$ as a test space:
\begin{align}\label{eq:2dlambdakepsilon}
\overline{\lambda}_{h,\beta}(P_\varepsilon)\le&\max_{\overset{\alpha^\varepsilon_1,\ldots,\alpha^\varepsilon_h\in\R}{\sum_i(\alpha_i^\varepsilon)^2=1}}\frac{\displaystyle\int_{P_\varepsilon}\Big|\sum_i\alpha^\varepsilon_i\nabla u_i\Big|^2\:dx+\beta\int_{\partial P_\varepsilon}\left[\Big(\sum_i\alpha_i^\varepsilon u_i^-\Big)^2+\Big(\sum_i\alpha_i^\varepsilon u_i^+\Big)^2\right]\:d\mathcal{H}^1}{\displaystyle\int_{P_\varepsilon}\Big(\sum_i\alpha_i^\varepsilon u_i\Big)^2\:dx}
\nonumber \\ 
\le&\frac{\displaystyle \int_{P}\Big|\sum_i\overline{\alpha}^\varepsilon_i\nabla u_i\Big|^2\:dx+\beta\int_{\partial P}\left[\Big(\sum_i\overline{\alpha}_i^\varepsilon u_i^-\Big)^2+\Big(\sum_i\overline{\alpha}_i^\varepsilon u_i^+\Big)^2\right]\:d\mathcal{H}^1}{1-\displaystyle\int_{m_\varepsilon}\Big(\sum_i\overline{\alpha}_i^\varepsilon u_i\Big)^2\:dx}
\\ \nonumber
&+\frac{\displaystyle \beta\int_{\sigma_\varepsilon}\Big(\sum_i\overline{\alpha}_i^\varepsilon u_i\Big)^2\:d\mathcal{H}^1-\beta\int_{s_\varepsilon}\Big(\sum_i\overline{\alpha}_i^\varepsilon u_i^-\Big)^2\:d\mathcal{H}^1}{1-\displaystyle\int_{m_\varepsilon}\Big(\sum_i\overline{\alpha}_i^\varepsilon u_i\Big)^2\:dx}
\\ \nonumber
\le&\overline{\lambda}_{h,\beta}(P)+\beta\int_{\sigma_\varepsilon}\Big(\sum_i\overline{\alpha}_i^\varepsilon u_i\Big)^2\:d\mathcal{H}^1-\beta\int_{s_\varepsilon}\Big(\sum_i\overline{\alpha}_i^\varepsilon u_i^-\Big)^2\:d\mathcal{H}^1+ C|m_\varepsilon|
\end{align}
where we have used that 
$$\overline{R}_{P,\beta}\left(\sum_i\overline{\alpha}_i^\varepsilon u_i\right)\le\max_{v\in S}\overline{R}_{P,\beta}(v)=\lambda_{h,\beta}(P).$$
Now, as highlighted at the beginning of Lemma \ref{lem:2dlambda1}, $\mathcal{H}^1(\sigma_\varepsilon)$ and $\mathcal{H}^1(s_\varepsilon)$ are both infinitesimal of order $\varepsilon$ and $|m_\varepsilon|$ is infinitesimal of order $\varepsilon^2$ as $\varepsilon\to 0$; moreover,  $\overline{\alpha}_i^\varepsilon-\overline{\alpha}_i\to 0$. Then, summing and subtracting the contribution of $\sum_i\overline{\alpha}_i u_i$ in the boundary integrals in  \eqref{eq:2dlambdakepsilon}, we get
\begin{equation}\label{eq:2dlambdakepsilon2}
\begin{split}
\overline{\lambda}_{h,\beta}(P_\varepsilon)\le&\overline{\lambda}_{h,\beta}(P)+\beta\Bigg[\int_{\sigma_\varepsilon}\Big(\sum_i(\overline{\alpha}_i^\varepsilon-\overline{\alpha}_i) u_i+\overline{\alpha}_i u_i\Big)^2\:d\mathcal{H}^1\\
&-\int_{s_\varepsilon}\Big(\sum_i(\overline{\alpha}_i^\varepsilon-\overline{\alpha}_i) u^-_i+\overline{\alpha}_i u_i^-\Big)^2\:d\mathcal{H}^1\Bigg]+ C\varepsilon^2\\
\le&\overline{\lambda}_{h,\beta}(P)+\beta\left(\int_{\sigma_\varepsilon}\Big(\sum_i\overline{\alpha}_i u_i\Big)^2\:d\mathcal{H}^1-\int_{s_\varepsilon}\Big(\sum_i\overline{\alpha}_i u_i^-\Big)^2\:d\mathcal{H}^1\right)+o(\varepsilon).
\end{split}
\end{equation}
To conclude, we now consider two possible situations. If
\begin{equation*}
\Big(\sum_i\overline{\alpha}_i u_i(0)\Big)^2\neq 0,
\end{equation*}
(where the pointwise value $u_i(0)$ is meant in the same sense as in Lemma \ref{lem:2dlambda1}) then, for any sufficiently small $\varepsilon$, we can proceed as in Lemma \ref{lem:2dlambda1} and conclude that
$$
\int_{\sigma_\varepsilon}\Big(\sum_i\overline{\alpha}_i u_i\Big)^2\:d\mathcal{H}^1-\int_{s_\varepsilon}\Big(\sum_i\overline{\alpha}_i u_i^-\Big)^2\:d\mathcal{H}^1\le 0.
$$
On the other hand, if
\begin{equation*}
\Big(\sum_i\overline{\alpha}_i u_i(0)\Big)^2=0,
\end{equation*}
the uniform continuity of the eigenfunctions $u_i$ on $m_\varepsilon$ implies that $\sum_i\overline{\alpha}_i u_i$ has values close to $\sum_i\overline{\alpha}_i u_i(0)=0$ in $m_\varepsilon$, namely that
$$\left|\sum_i\overline{\alpha}_i u_i\right|\le\delta(\varepsilon)$$
in $m_\varepsilon$, where $\delta(\varepsilon)\to 0$. Then, both boundary integrals are $o(\varepsilon)$ as $\varepsilon\to 0$. In both the possible situations, \eqref{eq:2dlambdakepsilon2} gives
$$
\overline{\lambda}_{h,\beta}(P_\varepsilon)\le\overline{\lambda}_{h,\beta}(P)+o(\varepsilon).
$$
\end{proof}
\end{lemma}

\begin{remark} 
Let us compare the results of the previous lemmas. In Lemma \ref{lem:2dlambda1} we proved that, after a small cut, the first eigenvalue decreases by a term of the same order as the perimeter. In Lemma \ref{lem:2dlambdak}, we proved that a small cut could increase $\lambda_{h,\beta}$ ($h\ge 2$) at most by a term infinitesimal of higher order than the perimeter. In other words, the possible increase of $\lambda_{h,\beta}$ ($h\ge 2$) is infinitesimal of higher order than the decrease of $\lambda_{1,\beta}$.
\end{remark}

In the following proposition we prove that the number of sides of possible optimal generalized polygons is maximal. 

\begin{proposition}\label{lem:sidenumber}
Let $F:\R^k\to\R$ satisfy the hypotheses of Problems \eqref{eq:poly} and \eqref{eq:polyper}. For every $P\in\overline{\mathcal{P}_{N}}$, there exists a generalized polygon $P'\in\overline{\mathcal{P}_{N+1}}$ such that $|P'|\le |P|$, $\widetilde{Per}(P')\le \widetilde{Per}(P)$ and
$$
F(\overline{\lambda}_{1,\beta}(P'),\ldots,\overline{\lambda}_{k,\beta}(P'))<F(\overline{\lambda}_{1,\beta}(P),\ldots,\overline{\lambda}_{k,\beta}(P)).
$$
In particular, no $P\in\overline{\mathcal{P}_{N}}$ can be a minimizer for \eqref{eq:poly} or \eqref{eq:polyper} in $\overline{\mathcal{P}_{N+1}}$.
\end{proposition}
\begin{proof}
The proof is based on an argument similar to that in the proof of \cite[Theorem 5.3]{citoconvex}.
Consider $P\in\overline{\mathcal{P}_N}$. For every $\varepsilon>0$ sufficiently small, the polygon $P_\varepsilon$ defined in \eqref{eq:2depsilonsets} has $N+1$ sides; moreover, Lemma \ref{lem:2dlambda1} and Lemma \ref{lem:2dlambdak} imply that
\begin{equation}\label{diftaglio}
\overline{\lambda}_{1,\beta}(P_\varepsilon)\le\overline{\lambda}_{1,\beta}(P)-C\varepsilon\quad\text{and}\quad \overline{\lambda}_{k,\beta}(P_\varepsilon)\le \overline{\lambda}_{k,\beta}(P)+ o(\varepsilon).
\end{equation}
The hypotheses on $F$ lead to the first assertion, once we set $P':=P_\varepsilon\in\overline{\mathcal{P}_{N+1}}$ for a suitable value of $\varepsilon>0$. Indeed, a Taylor expansion gives
\begin{equation}\label{eq:splittaglio}
\begin{split}
F(\overline\lambda_{1,\beta}(P_\varepsilon),\ldots,\overline\lambda_{k,\beta}(P_\varepsilon))
=&F(\overline\lambda_{1,\beta}(P),\ldots,\overline\lambda_{k,\beta}(P))\\
&+\sum_{h=1}^k\frac{\partial F}{\partial x_h}(\overline\lambda_{1,\beta}(P),\ldots,\overline\lambda_{k,\beta}(P))\cdot(\lambda_{h,\beta}(P_\varepsilon)-\overline\lambda_{h,\beta}(P))\\
&+o\left(\left|\left(\lambda_{h,\beta}(P_\varepsilon)-\overline\lambda_{h,\beta}(P)\right)_{h=1,\ldots,k}\right|\right).
\end{split}
\end{equation}
Now, using \eqref{diftaglio} yields
\begin{align*}
\left|\left(\lambda_{h,\beta}(P_\varepsilon)-\overline\lambda_{h,\beta}(P)\right)_{h=1,\ldots,k}\right|&=\sqrt{\left(\lambda_{1,\beta}(P_\varepsilon)-\overline\lambda_{1,\beta}(P)\right)^2+\ldots+\left(\lambda_{k,\beta}(P_\varepsilon)-\overline\lambda_{k,\beta}(P)\right)^2}\\
&=\sqrt{\varepsilon^2+o(\varepsilon)^2}=\varepsilon+o(\varepsilon)
\end{align*}
Plugging the latter in \eqref{eq:splittaglio} and using again \eqref{diftaglio} on each term in the sum at the third line we finally get
\begin{align*}
F(\overline\lambda_{1,\beta}&(P_\varepsilon),\ldots,\overline\lambda_{k,\beta}(P_\varepsilon))  \\
&\le F(\overline\lambda_{1,\beta}(P),\ldots,\overline\lambda_{k,\beta}(P))-\frac{\partial F}{\partial x_1}(\overline\lambda_{1,\beta}(P),\ldots,\overline\lambda_{k,\beta}(P))\cdot(C\varepsilon)+o(\varepsilon)\\
&=F(\overline\lambda_{1,\beta}(P),\ldots,\overline\lambda_{k,\beta}(P))-C'\varepsilon+o(\varepsilon)<F(\overline\lambda_{1,\beta}(P),\ldots,\overline\lambda_{k,\beta}(P)).
\end{align*}
In particular, if we consider a generalized polygon $P\in\overline{\mathcal{P}_{N}}\subset\overline{\mathcal{P}_{N+1}}$, the corresponding generalized polygon $P'$ (built as above) gives us a strictly lower value for Problem \eqref{eq:poly} in $\overline{\mathcal{P}_{N+1}}$, then $P$ cannot be a minimizer in $\overline{\mathcal{P}_{N+1}}$.
\end{proof}

In order to prove
Main Theorem \ref{mtheorem:main1}, we focus for the moment only on Problem \eqref{eq:poly1}, since the generalized perimeter constraint in Problem \eqref{eq:polyper} provides extra compactness. It seems that removing the bounded design region hypothesis in the measure constrained case of Problem \eqref{eq:poly} is not possible, since the detection of two disjoint polygonal components in the dichotomy case of the concentration-compactness argument similar to those in \cite{bucgia19,citogiaco}, could create extra sides more than $N$.
We first prove the existence of an optimal polygon for $\overline{\lambda}_{1,\beta}$.
\begin{proposition}
\label{lem:lambda1}
    Problem \eqref{eq:poly1} admits a solution. Any solution is a connected polygon with exactly $N$ sides. 
\end{proposition}

\begin{proof}
Let us consider a minimizing sequence $(P_n)_n\subset\overline{\mathcal{P}_{N}}$ for Problem \eqref{eq:poly1}. Without loss of generality we can assume $|P_n|=m$ and $\overline{\lambda}_{1,\beta}(P_n)<\Lambda$ for some $\Lambda>0$. Since $\R^2\setminus P_n$ has a uniformly bounded number of connected components, we deduce from Remark \ref{pro:compoly}(i) that there exists $P\in\overline{\mathcal{P}_N}$ such that
$$
P_n\xrightarrow{H^c_{loc}} P
$$
and locally in measure as well.

We claim that $|P|=m$. In order to do this, we apply a concentration-compactness argument to the sequence $(\chi_{P_n})_{n\in\N}$, where $\chi_{P_n}$ is the indicator function of $P_n$. For every $r>0$ let us consider the monotone increasing functions $\alpha_n:[0,+\infty[\to [0,+\infty[$
\begin{equation*}
\alpha_n(r):=\sup_{y \in \R^N}|P_n \cap Q_r(y)|,
\end{equation*}
where $Q_r(y)$ is the square centered at $y$ with side $r$. Up to a subsequence, in view of Helly's selection  theorem, we may assume that
\begin{equation*}
\alpha_n \to \alpha
\qquad\text{pointwise on }[0,+\infty[
\end{equation*}
for a suitable monotone increasing function $\alpha:[0,+\infty[\to [0,+\infty[$. 
Only the three following situations may occur.
\begin{itemize}
\item[(a)] {\it Vanishing}: $\lim_{r\to+\infty}\alpha(r)=0$;
\item[(b)] {\it Dichotomy}: $\lim_{r\to+\infty}\alpha(r)=\bar \alpha \in ]0,m[$;
\item[(c)] {\it Compactness}: $\lim_{r\to+\infty}\alpha(r)=m$.
\end{itemize}
Let us show that, for our problem, only situation $(c)$ can occur. 
\vskip10pt\noindent{\bf Step 1: Vanishing cannot occur.} Indeed, if it were the case, one would have for every $r>0$
\begin{equation}
\label{eq:starvanish}
\sup_{y\in\R^N}\left|P_n\cap Q_r(y)\right|\to 0
\end{equation}
as $n\to+\infty$.
Let $u_n$ be a $L^2$-normalized eigenfunction of $P_n$. The uniform estimate
$$
\int_{P_n}|\nabla u_n|^2\:dx+\beta\int_{\partial P_n}\left[(u_n^-)^2+(u_n^+)^2\right]\:d\mathcal{H}^1\le \Lambda,
$$
induces the following
$$
\int_{P_n}|\nabla u_n|^2\:dx+\int_{\partial P_n}\left[(u_n^-)^2+(u_n^+)^2\right]\:d\mathcal{H}^1\le\Lambda\left(1+\frac1\beta\right).
$$
So, in view of Lemma \ref{lem:conc}, there exists $y\in\R^2$ such that
$$
|P_n\cap Q_1(y)|\ge C \quad \forall n\in\N
$$
for some uniform constant $C>0$, in contradiction with \eqref{eq:starvanish}.
\vskip10pt\noindent{\bf Step 2: Dichotomy cannot occur.} 
Let us assume that dichotomy occurs. Then there exists $\tilde\alpha\in]0,m[$ such that the following assertion holds true: we can find $x_n\in\R^2$ and $0<r_n<R_n$, $R_n-r_n\to+\infty$, such that we have
$$
\left||P_n\cap Q_{r_n}(x_n)|-\tilde\alpha\right|\to 0,\qquad \left||P_n\setminus Q_{R_n}(x_n)|-(m-\tilde\alpha)\right|\to 0,
$$
with
$$
\mathcal{H}^1(P_n\cap\partial Q_{r_n}(x_n))\to 0,\qquad \mathcal{H}^1(P_n\cap\partial Q_{R_n}(x_n))\to 0.
$$
Now, let us define $P_{n,1}$ and $P_{n,2}$ as follows. If the "transition region" $P_n\cap(Q_{R_n}(x_n)\setminus \overline{Q_{r_n}}(x_n))$ is empty, we set $P_{n,1}:=P_n\cap Q_{r_n}(x_n)$ and $P_{n,2}:=P_n\setminus Q_{R_n}(x_n)$. If not, let us consider the connected components $A_1,\ldots, A_k$ of the open set $P_n\cap(Q_{R_n}(x_n)\setminus \overline{Q_{r_n}}(x_n))$. Without loss of generality, we can assume that $\partial A_j\cap\partial Q_{r_n}(x_n))$ and $\partial A_j\cap\partial Q_{R_n}(x_n))$ consist of only one segment each.
We now remove "channels" that connect plain parts as in Figure \ref{fig:dic} below.

\begin{figure}[h!]
\begin{tikzpicture}[>=>>>]
\label{fig:dic}
\fill[purple!60!] (-2,0) .. controls (-2,0) and (-3,1) ..
(-3,1) .. controls (-3,1) and (-3,3) ..
(-3,3) .. controls (-3,3) and (0,4) .. 
(0,4) .. controls (0,4) and (8,4) .. 
(8,4) .. controls (8,4) and (0.5,3.75) ..
(0.5,3.75).. controls (0.5,3.75) and (0.5,3) ..
(0.5,3) .. controls (0.5,3) and (7.5,2.9) .. 
(7.5,2.9) .. controls (7.5,2.9) and (8.5,4) .. 
(8.5,4) .. controls (8.5,4) and (10,3.5) .. 
(10,3.5) .. controls (10,3.5) and (10,3) .. 
(10,3) .. controls (10,3) and (9,2) ..
(9,2) .. controls (9,2) and (8,2) ..
(8,2) .. controls (8,2) and (7,2.75) ..
(7,2.75) .. controls (7,2.75) and (1,2.9) ..
(1,2.9) .. controls (1,2.9) and (.6,2) ..
(.6,2) .. controls (.6,2) and (2.5,1) ..
(2.5,1) .. controls (2.5,1) and (4,.75) ..
(4,.75) .. controls (4,.75) and (8,1.5) ..
(8,1.5) .. controls (8,1.5) and (2,2) ..
(2,2) .. controls (2,2) and (8,1.8) ..
(8,1.8) .. controls (8,1.8) and (9,1.5) ..
(9,1.5) .. controls (9,1.5) and (10.5,3) ..
(10.5,3) .. controls (10.5,3) and (10.5,4) ..
(10.5,4) .. controls (10.5,4) and (12,4) ..
(12,4) .. controls (12,4) and (12,0) ..
(12,0) .. controls (12,0) and (8,0) ..
(8,0) .. controls (8,0) and (7,1) ..
(7,1) .. controls (7,1) and (3,0.5) ..
(3,0.5) .. controls (3,0.5)  and (1,1.5) ..
(1,1.5) .. controls (1,1.5) and (0,0) ..
(0,0) .. controls (0,0) and (-2,0) ..
(-2,0);
\fill[white]
(1,2) .. controls (1,2) and (7,1.7) .. 
(7,1.5) .. controls (7,1.5) and (7,0) .. 
(7,0) .. controls (7,0) and (1,0) ..
(1,0) .. controls (1,0) and (1,2) ..
(1,2);
\fill[white]
(1,3) .. controls (1,3) and (7,3) .. 
(7,3) .. controls (7,3) and (7,2.5) .. 
(7,2.5) .. controls (7,2.5) and (1,2.5) ..
(1,2.5) .. controls (1,2.5) and (1,3) ..
(1,3);
\draw[line width=1.1pt]  
(-2,0) .. controls (-2,0) and (-3,1) ..
(-3,1) .. controls (-3,1) and (-3,3) ..
(-3,3) .. controls (-3,3) and (0,4) .. 
(0,4) .. controls (0,4) and (8,4) .. 
(8,4) .. controls (8,4) and (0.5,3.75) ..
(0.5,3.75).. controls (0.5,3.75) and (0.5,3) ..
(0.5,3) .. controls (0.5,3) and (7.5,2.9) .. 
(7.5,2.9) .. controls (7.5,2.9) and (8.5,4) .. 
(8.5,4) .. controls (8.5,4) and (10,3.5) .. 
(10,3.5) .. controls (10,3.5) and (10,3) .. 
(10,3) .. controls (10,3) and (9,2) ..
(9,2) .. controls (9,2) and (8,2) ..
(8,2) .. controls (8,2) and (7,2.75) ..
(7,2.75) .. controls (7,2.75) and (1,2.9) ..
(1,2.9) .. controls (1,2.9) and (.6,2) ..
(.6,2) .. controls (.6,2) and (2.5,1) ..
(2.5,1) .. controls (2.5,1) and (4,.75) ..
(4,.75) .. controls (4,.75) and (8,1.5) ..
(8,1.5) .. controls (8,1.5) and (2,2) ..
(2,2) .. controls (2,2) and (8,1.8) ..
(8,1.8) .. controls (8,1.8) and (9,1.5) ..
(9,1.5) .. controls (9,1.5) and (10.5,3) ..
(10.5,3) .. controls (10.5,3) and (10.5,4) ..
(10.5,4) .. controls (10.5,4) and (12,4) ..
(12,4) .. controls (12,4) and (12,0) ..
(12,0) .. controls (12,0) and (8,0) ..
(8,0) .. controls (8,0) and (7,1) ..
(7,1) .. controls (7,1) and (3,0.5) ..
(3,0.5) .. controls (3,0.5)  and (1,1.5) ..
(1,1.5) .. controls (1,1.5) and (0,0) ..
(0,0) .. controls (0,0) and (-2,0) ..
(-2,0);
\draw[line width=1.1pt]  (-1,0) .. controls (-1,0) and (0,2) ..(0,2) .. controls (0,2) and (-1,2.5) ..(-1,2.5);
\draw[line width=1.1pt]   (6.5,1.1) .. controls (6.5,1.1) and (8.5,1) .. (8.5,1) .. controls (8.5,1) and (9,0.5) ..(9,0.5) .. controls (9,0.5) and (10,2.5) ..(10,2.5);
\draw[line width=.7pt,dashed] (1,-.5) .. controls (1,-.5) and (1,5) ..(1,5);
\draw[line width=.7pt,dashed] (7,-.5) .. controls (7,-.5) and (7,5) ..(7,5);
\draw (.5,4.3) node[anchor=south] {$Q_{r_n}$};
\draw (6.5,4.3) node[anchor=south] {$Q_{R_n}$};
\draw (-2,3.7) node[anchor=south] {$P_{n,1}$};
\draw (10,3.7) node[anchor=south] {$P_{n,2}$};
\draw[->,dashed] (-2,3.7) -- (-1.5,2.6);
\draw[->,dashed] (10,3.7) -- (10.5,2.6);
\draw[->,dashed] (10,3.7) -- (9.5,2.8);
\end{tikzpicture}
\caption{$P_n$ is the generalized polygon with the black contour; $P_{n,1}$ and $P_{n,2}$ are detected in color as in the picture.}
\end{figure}

More precisely, we take a connected component $A_j$ and analyze the adjacent connected polygons: if both have measure tending to a positive number, then we remove $A_j$, otherwise we keep it in. Let us assume that such removed channels are $A_1,\ldots, A_l$. We obtain a disconnected polygon
$$D_n:=P_n\setminus \bigcup_{j=1}^l \overline{A_j}.$$
By construction, for every connected component $D_{j,n}$ of $D_n$ it holds
\begin{equation}\label{eq:pn1}
|D_{j,n}\cap Q_{r_n}(x_n)|\to c_j>0\ \Rightarrow\ |D_{j,n}\setminus Q_{R_n}(x_n)|\to0    
\end{equation}
and
\begin{equation}\label{eq:pn2}
|D_{j,n}\setminus Q_{R_n}(x_n)|\to c_j>0\ \Rightarrow\ |D_{j,n}\cap Q_{r_n}(x_n)|\to0.    
\end{equation}
We denote by $P_{n,1}$ and $P_{n,2}$ respectively the union of the connected components of $D_n$ satisfying \eqref{eq:pn1} and \eqref{eq:pn2}.

Notice that both $P_{n,1}$ and $P_{n,2}$ belong to $\overline{\mathcal{P}_N}$. Indeed, for instance, $P_{n,1}$ is obtained removing from $P_n$ the $l$ channels and the $l$ adjacent connected polygons associated with $P_{n,2}$; each discarded component has uniformly nondegenerate measure. At each step, we remove the extra sides of the connected component, that has at least three sides; notice that at least two of these sides lie outside $Q_{r_n}(x_n)$ (one side can be a long segment starting from $Q_{r_n}(x_n)$ and so already counted in the number of sides of $P_{n,1}$). On the other hand, at each step we add at most one straight cutting segment (cap) on the boundary of $Q_{r_n}(x_n)$. So the total number of sides of $P_{n,1}$ is strictly lower than $N$; the same holds for $P_{n,2}$.

Moreover, both polygons have positive measure. In particular $|P_{n,1}|\to\tilde{\alpha}$ and $|P_{n,2}|\to m-\tilde{\alpha}$. Furthermore
$$\mathcal{H}^1\left(\partial P_{n,1}\cap\partial Q_{R_n}(x_n)\right)\to 0,\quad \mathcal{H}^1\left(\partial P_{n,2}\cap\partial Q_{R_n}(x_n)\right)\to 0.$$
Let $u_n$ be a $L^2$-normalized eigenfunction of $P_n$. By exploiting the elementary algebraic property of positive fractions 
$$\frac{a+b}{c+d}\ge\min\left\{\frac{a}{c},\frac{b}{d}\right\},$$
we can bound the Rayleigh quotient from below as follows (possibly switching $P_{n,1}$ and $P_{n,2}$):
\allowdisplaybreaks
\begin{equation}\label{eq:dicotomialungo}
\begin{split}
\overline{\lambda}&_{1,\beta}(P_n)=\frac{\displaystyle\int_{P_n}|\nabla u_n|^2\:dx+\beta\int_{\partial P_n}\left[(u_n^-)^2+(u_n^+)^2\right]\:d\mathcal{H}^1}{\displaystyle\int_{P_n} u_n^2\:dx}
\\
=&\left(\int_{P_{n,1}} u_n^2\:dx+\int_{P_{n,2}} u_n^2\:dx+\int_{P_n\setminus(P_{n,1}\cup P_{n,2})}u_n^2\:dx\right)^{-1}\\
&\cdot\left(\displaystyle\int_{P_{n,1}}|\nabla u_n|^2\:dx+\int_{P_{n,2}}|\nabla u_n|^2\:dx+\int_{P_n\setminus(P_{n,1}\cup P_{n,2})}|\nabla u_n|^2\:dx+\beta\int_{\partial P_{n,1}}\left[(u_n^-)^2+(u_n^+)^2\right]\:d\mathcal{H}^1\right.
\\
&+\beta\int_{\partial P_{n,2}}\left[(u_n^-)^2+(u_n^+)^2\right]\:d\mathcal{H}^1+\beta\int_{\partial P_n\setminus(P_{n,1}\cup P_{n,2})}\left[(u_n^-)^2+(u_n^+)^2\right]\:d\mathcal{H}^1
\\
&\left.-\beta\int_{\partial P_{n,1}\cap\partial Q_{r_n}(x_n)}\left[(u_n^-)^2+(u_n^+)^2\right]\:d\mathcal{H}^1-\beta\int_{\partial P_{n,2}\cap\partial Q_{R_n}(x_n)}\left[(u_n^-)^2+(u_n^+)^2\right]\:d\mathcal{H}^1\right)
\\
\ge&\left(\int_{P_{n,1}} u_n^2\:dx+\int_{P_{n,2}} u_n^2\:dx+\varepsilon_n\right)^{-1}\\
&\cdot\left(\int_{P_{n,1}}|\nabla u_n|^2\:dx+\beta\int_{\partial P_{n,1}}\left[(u_n^-)^2+(u_n^+)^2\right]\:d\mathcal{H}^1\right.\\
&\left.+\int_{P_{n,2}}|\nabla u_n|^2\:dx+\beta\int_{\partial P_{n,2}}\left[(u_n^-)^2+(u_n^+)^2\right]\:d\mathcal{H}^1-\delta_n\right)
\\
\ge&\left(\frac{\displaystyle\int_{P_{n,1}}|\nabla u_n|^2\:dx+\beta\int_{\partial P_{n,1}}\left[(u_n^-)^2+(u_n^+)^2\right]\:d\mathcal{H}^1}{\displaystyle\int_{P_{n,1}} u_n^2\:dx+\int_{P_{n,2}} u_n^2\:dx}\right.
\\
&+\left.\frac{\displaystyle\int_{P_{n,2}}|\nabla u_n|^2\:dx+\beta\int_{\partial P_{n,2}}\left[(u_n^-)^2+(u_n^+)^2\right]\:d\mathcal{H}^1}{\displaystyle\int_{P_{n,1}} u_n^2\:dx+\int_{P_{n,2}} u_n^2\:dx}-\delta_n\right)\cdot(1-\varepsilon_n)
\\
\ge&\left(\min\left\{\overline{R}_{P_{n,1},\beta}(u_n),\overline{R}_{P_{n,2},\beta}(u_n)\right\}-\delta_n\right)\cdot(1-\varepsilon_n)\ge\left(\min\left\{\overline{\lambda}_{1,\beta}(P_{n,1}),\overline{\lambda}_{1,\beta}(P_{n,2})\right\}-\delta_n\right)\cdot(1-\varepsilon_n)\\
=&\left(\overline{\lambda}_{1,\beta}(P_{n,1})-\delta_n\right)\cdot(1-\varepsilon_n)
\end{split}
\end{equation}
with $\varepsilon_n,\delta_n\to0$. Now, let us consider the polygons
$$
\tilde{P}_{n,1}:=\sqrt{\frac{m}{|P_{n,1}|}} P_{n,1}
$$
whose measure is $m>|P_{n,1}|$. In view of \eqref{eq:scal} we have
\begin{equation*}
    \overline{\lambda}_{1,\beta}(P_{n,1})>\sqrt{\frac{m}{|P_{n,1}|}}\overline{\lambda}_{1,\beta}(\tilde{P}_{n,1}).
\end{equation*}
Plugging in \eqref{eq:dicotomialungo} we get
\begin{equation*}
\overline{\lambda}_{1,\beta}(P_n)>\left(\sqrt{\frac{m}{|P_{n,1}|}}\overline{\lambda}_{1,\beta}(\tilde{P}_{n,1})-\delta_n\right)\cdot(1-\varepsilon_n)\ge\left(\sqrt{\frac{m}{|P_{n,1}|}}\inf_{P\in\overline{\mathcal{P}_N},|P|\le m}\overline{\lambda}_{1,\beta}(P)-\delta_n\right)\cdot(1-\varepsilon_n).
\end{equation*}
Since $P_n$ is a minimizing sequence, passing to the limit as $n\to+\infty$ we get
$$
\inf_{P\in\overline{\mathcal{P}_N},|P|\le m}\overline{\lambda}_{1,\beta}(P)=\lim_{n\to+\infty}\overline{\lambda}_{1,\beta}(P_n)\ge\sqrt{\frac{m}{\tilde{\alpha}}}\inf_{P\in\overline{\mathcal{P}_N},|P|\le m}\overline{\lambda}_{1,\beta}(P),
$$
getting a contradiction since the infimum is positive and $\sqrt{\frac{m}{\tilde{\alpha}}}>1$.

\vskip10pt\noindent{\bf Step 3: Compactness.} Since neither vanishing nor dichotomy occur, then compactness occurs. So, the $H^c_{loc}$-limit generalized polygon $P$ has measure $m$. This implies that $P$ is bounded. Let us fix $R>0$ such that $P\subset Q_R$. Without loss of generality we can choose $R$ large enough in such a way that, up to subsequences
$$
\lim_{n\to+\infty}\mathcal{H}^1\left( P_{n}\cap\partial Q_R\right)=0,\quad \lim_{n\to+\infty}|P_{n}\setminus Q_R|=0.
$$
Let us consider the sequence of open generalized polygons $P_{n,R}:=P_n\cap Q_R$. Notice that, since the number of sides of $P_{n,R}$ is uniformly bounded and each side has maximal length $\sqrt{2}R$ (the diagonal of $Q_R$), it holds
$$
\sup_n\mathcal{H}^1\left(\partial P_{n,R}\right)<+\infty.
$$
Moreover, it holds that
$$
P_{n,R}\overset{H^c}{\longrightarrow} P\quad\text{and}\quad|P_{n,R}|\to|P|.
$$
So, thanks to Theorem \ref{prop:Mosco-domains}, $H^1(P_{n,R})$ converges in the sense of Mosco to $H^1(P)$.
Let us consider the first $L^2(P_n)$-normalized eigenfunction $u_n$ of $P_n$ and denote by $u_{n,R}$ its restriction to $P_{n,R}$ extended outside by zero. In view of the convergence in the sense of Mosco of $H^1(P_{n,R})$ to $H^1(P)$, there exists $u\in H^1(P)$ such that
$$u_{n,R}\to u\quad\text{strongly in $L^2(\R^2)$},$$
$$\chi_{P_n}\nabla u_{n,R}\rightharpoonup \chi_P\nabla u\quad\text{weakly in $L^2(\R^2;\R^2)$}$$
and so
\begin{equation}\label{eq:lsccutvolumi}
\int_Pu^2\:dx=\lim_{n\to+\infty}\int_{P_{n,R}}u_{n,R}^2\:dx,\quad \int_P|\nabla u|^2\:dx\le\liminf_{n\to+\infty}\int_{P_{n,R}}|\nabla u_{n,R}|^2\:dx.
\end{equation}
We recall that, in view of Theorem \ref{th:boundeigen}, $\|u_n\|_\infty<C$ for some positive constant $C$ independent of $n$ and $R$, so $\|u_{n,R}\|_\infty<C$ as well. We can apply Lemma \ref{teo:lscboundary} in the open set $Q_{R+1}$ with the sequences of compact sets $\partial P_{n,R}$. We consider a subsequence, not relabeled, $u_{n,R}\in H^1(Q_{R+1}\setminus \partial P_{n,R})$ such that its $L^2$-limit is $u$ in view of the Mosco-convergence of the Sobolev spaces $H^1$. It holds
$$\|u_{n,R}\|_{H^1(Q_{R+1}\setminus \partial P_{n,R})} +\int_{\partial P_{n,R}}\left[(u_n^+)^2+(u_n^-)^2\right]\:d\mathcal{H}^1<C$$
(the latter boundary integral being bounded thanks to the $L^\infty$ uniform bound) and $\partial P$ is contained in every Hausdorff limit of every converging subsequence of $\partial P_{n,R}$. So, we get that \eqref{eq:lscboundary} holds:
\begin{equation}\label{eq:lsccut}
    \int_{\partial P}\left[(u^+)^2+(u^-)^2\right]\:d\mathcal{H}^1\le\liminf_{n\to+\infty}\int_{\partial P_{n,R}}\left[(u_{n,R}^+)^2+(u_{n,R}^-)^2\right]\:d\mathcal{H}^1.
\end{equation}
Notice that
$$\partial P_{n,R}=(Q_R\cap\partial P_n)\cup(P_n\cap \partial Q_R)\subset \partial P_n\cup(P_n\cap \partial Q_R).$$
We thus get
\begin{align*}
    \overline{\lambda}_{1,\beta}(P)&\le\frac{\int_P|\nabla u|^2\:dx+\beta\int_{\partial P}\left[(u^+)^2+(u^-)^2\right]\:d\mathcal{H}^1}{\int_Pu^2\:dx}\\
    &\le\liminf_{n\to+\infty}\frac{\int_{P_{n,R}}|\nabla u_{n,R}|^2\:dx+\beta\int_{\partial P_{n,R}}\left[(u_{n,R}^+)^2+(u_{n,R}^-)^2\right]\:d\mathcal{H}^1}{\int_{P_{n,R}}u_{n,R}^2\:dx}\\
    &\le\liminf_{n\to+\infty}\frac{\int_{P_{n}}|\nabla u_{n}|^2\:dx+\beta\int_{\partial P_{n}}\left[(u_{n}^+)^2+(u_{n}^-)^2\right]\:d\mathcal{H}^1+\beta\int_{P_n\cap \partial Q_R}\left[(u_{n}^+)^2+(u_{n}^-)^2\right]\:d\mathcal{H}^1}{1-\int_{P_n\setminus Q_{R}}u_{n,R}^2\:dx}\\
    &\le\liminf_{n\to+\infty}\frac{\overline{\lambda}_{1,\beta}(P_n)+2\beta\|u_n\|_\infty^2\mathcal{H}^1\left(P_{n}\cap\partial Q_R\right)}{1-\|u_n\|_\infty^2|P_n\setminus Q_{R}|}=\lim_{n\to+\infty}\overline{\lambda}_{1,\beta}(P_n)=\inf\eqref{eq:poly1}.
\end{align*}
We infer that $P$ is a minimizer for Problem \eqref{eq:poly1}.
\end{proof}

We are now in a position to complete the proof our Main Theorem \ref{mtheorem:main1}.

\begin{proof}
Let us start observing that in the case of problem \eqref{eq:polyper} with perimeter constraint, the generalized perimeter constraint entails a uniform bound on the diameters. So there exists a compact $D$ such that every admissible polygon is contained in $D$. So, the proof follows the same pattern for both problems \eqref{eq:poly} and \eqref{eq:polyper} and we present it only once, for the volume constrained case.

Let us consider a minimizing sequence $(P_n)_n$ for \eqref{eq:poly}. Without loss of generality, we assume that
\begin{equation}\label{eq:upbound}
F(\overline{\lambda}_{1,\beta}(P_n),\ldots,\overline{\lambda}_{k,\beta}(P_n))\le\inf_{P\in\overline{\mathcal{P}_N},|P|\le m}F(\overline{\lambda}_{1,\beta}(P),\ldots,\overline{\lambda}_{k,\beta}(P))+1.
\end{equation}

Now, since $P_n\subset D$, there exists, up to a subsequence, a generalized polygon $P\in\overline{\mathcal{P}_N}$, $P\subset D$, such that $P_n$ $H^c$-converges to $P$.

It is immediate to see that $P\neq \emptyset$, using Lemma \ref{lem:conc} as in the "vanishing cannot occur" case of Lemma \ref{lem:lambda1}. Moreover, we also have that $|P_n|\to |P|$ and $\sup\mathcal{H}^1(\partial P_n)<+\infty$ since $P_n$ has at most $N$ sides of length $\text{diam}(D)$. This entails the convergence in the sense of Mosco of $H^1(P_n)$ to $H^1(P)$, as well.

We now claim that
\begin{equation}\label{eq:semik}
\overline{\lambda}_{h,\beta}(P)\le\liminf_{n\to+\infty}\overline{\lambda}_{h,\beta}(P_n)\quad\forall h=1,\ldots,k.
\end{equation}
To show it, we consider a $L^2(P_n)$-orthonormal set, say $\left\{u^n_1,\ldots,u^n_h\right\}$, where $u^n_j$ is an eigenfunction for $\overline{\lambda}_{j,\beta}(P_n)$, and denote by $V_n$ the $h$-dimensional subspace of $H^1(P_n)$ generated. Since $H^1(P_n)$ converges to $H^1(P)$ in the sense of Mosco we conclude that, up to a further subsequence, for every $j=1,\ldots,h$ there exist $u_j\in H^1(P)$ such that $u_j^n\to u_j$ strongly in $L^2(\R^2)$ and $\chi_{P_n}\nabla u_j^n\rightharpoonup\chi_P\nabla u_j$ weakly in $L^2(\R^2;\R^2)$. Let $V$ be the $h$-dimensional vector space spanned by $\left\{u_1,\ldots,u_h\right\}$ (in view of the $L^2$-convergence we can suppose the $u_j$'s are linearly independent) and let us consider $v:=\sum_{j=1}^h\alpha_j u_j$ such that
$$
\overline{R}_{P,\beta}(v)=\max_{w\in V}\overline{R}_{P,\beta}(w).
$$
Let us consider $v_n:=\sum_{j=1}^h\alpha_j u^n_j\in V_n$ and observe that $v_n\to v$ strongly in $L^2(\R^2)$ and $\chi_{P_n}\nabla v_n\rightharpoonup\chi_P\nabla v$ weakly in $L^2(\R^2;\R^2)$. Thanks to the continuity of the volume integrals at the denominator and to the lower semicontinuity of the gradient integral and of the boundary integral (obtained in the same way as in Lemma \ref{lem:lambda1}), we obtain
\begin{align*}
\overline{\lambda}_{h,\beta}(P)&\le\max_{w\in V}\overline{R}_{P,\beta}(w)=\overline{R}_{P,\beta}(v)\le\liminf_{n\to+\infty}\overline{R}_{P_n,\beta}(v_n)
\\
&\le\liminf_{n\to+\infty}\max_{w\in V_n}\overline{R}_{P_n\beta}(w)=\liminf_{n\to+\infty}\overline{\lambda}_{h,\beta}(P_n),
\end{align*}
i.e. \eqref{eq:semik}. The hypotheses on $F$ immediately imply that $P$ is a minimizer for Problem \eqref{eq:poly}.

Finally, $P$ has exactly $N$ sides and $|P|=m$. Indeed, if it had less than $N$ sides, say $N-K$ sides, we could apply $K$ times Lemma \ref{lem:sidenumber} to $P$ to obtain a polygon $P'$ with exactly $N$ sides, $|P'|\le m$ and
$$F(\overline{\lambda}_{1,\beta}(P'),\ldots,\overline{\lambda}_{k,\beta}(P'))<F(\overline{\lambda}_{1,\beta}(P),\ldots,\overline{\lambda}_{k,\beta}(P)),$$
contradicting the minimality of $P$; the saturation of the volume constraint is assured thanks to the decreasing monotonicity under dilations \eqref{eq:scal}. We also deduce that no minimizer in $\overline{\mathcal{P}_{N}}$ can be a minimizer in $\overline{\mathcal{P}_{N+1}}$ and this implies that the sequence of minima $(m_N)$ is strictly decreasing.

In the same way we get the existence of a minimizer $P$ for Problem \eqref{eq:polyper} that saturates the generalized perimeter constraint and the maximal number of sides $N$.
\end{proof}

\begin{remark}
One can reasonably ask whether Problem \eqref{eq:poly} can be rephrased removing the bounded design region hypothesis. The technical problem is in the dichotomy case of the concentration-compactness argument. Indeed, the cutting argument in Proposition \ref{lem:lambda1} creates two disjoint polygons $P_1, P_2$ with at most $N$ sides each, but their union could exceed $N$ sides. For the first eigenvalue this is not a problem, since only one between $P_1$ and $P_2$ realizes the minimum, the other polygon being discarded. On the other hand, if $k>1$, \eqref{eq:compconn} implies that both disjoint parts should be taken into account (deleting one of them we would contradict the minimality of their union); this would lead to consider a minimal polygon that in general could have more than $N$ sides, so not admissible.
\end{remark}

As a corollary of the results of the section, the following problems
\begin{equation}\label{eq:polyconvex}
\min\left\{F(\lambda_{1,\beta}(P),\ldots,\lambda_{k,\beta}(P)):P\in\mathcal{P}_{N},|P|\le m,\ P\ \text{convex}\right\},
\end{equation}
\begin{equation}\label{eq:polyconvexper}
\min\left\{F(\lambda_{1,\beta}(P),\ldots,\lambda_{k,\beta}(P)):P\in\mathcal{P}_{N},\mathcal{H}^1(\partial P)\le p,\ P\ \text{convex}\right\}
\end{equation}
admit a solution. Notice that, in view of the convexity hypotheses, it is not necessary to consider also degenerate polygons. An existence result is obtainable as a corollary to Theorem \ref{mtheorem:main1}.
\begin{corollary}
Problem \eqref{eq:polyconvex} (respectively \eqref{eq:polyconvexper}) admits a solution with exactly $N$ sides and with maximal measure (respectively with maximal boundary length).
\begin{proof}
The proof is a combination of Theorem \ref{mtheorem:main1} and of the preservation of the convexity constraint under $H^c$-convergence.
\end{proof}
\end{corollary}

\medskip

We conclude this section with some remarks about the model problems 
\begin{equation}\label{eq:poly1classic}
\begin{split}
\min\Bigl\{\lambda_{1,\beta}(P):\ 
&P\ \text{union of simple polygons}, 
\\
&P \text{ has at most $N$ sides},|P|\le m\Bigr\},
\end{split}
\end{equation}
\begin{equation}\label{eq:polyclassic}
\begin{split}
\min\Bigl\{F(\lambda_{1,\beta}(P),\ldots,\lambda_{k,\beta}(P)):\ 
&P\ \text{union of simple polygons}, 
\\
&P \text{ has at most $N$ sides},|P|\le m, P\subseteq D\Bigr\}
\end{split}
\end{equation}
and
\begin{equation}\label{eq:polyclassicper}
\begin{split}
\min\Bigl\{F(\lambda_{1,\beta}(P),\ldots,\lambda_{k,\beta}(P)):\ &P\ \text{union of simple polygons}, 
\\
&P \text{ has at most $N$ sides},\mathcal{H}^1(\partial P)\le p,\Bigr\}.
\end{split}
\end{equation}
Indeed, in these case we are not allowed to apply the direct methods of the calculus of variation. That is the reason why we focused our analysis of their generalized versions \eqref{eq:poly1}, \eqref{eq:poly} and \eqref{eq:polyper}, respectively. Nevertheless, an easy approximation argument shows that \eqref{eq:poly1}, \eqref{eq:poly} and \eqref{eq:polyper} can be seen as a sort of relaxation of \eqref{eq:poly1classic}, \eqref{eq:polyclassic} and \eqref{eq:polyclassicper}.
\begin{proposition}
For any $P\in\overline{\mathcal{P}_N}$ admissible for Problem \eqref{eq:poly1} (resp. \eqref{eq:poly} or \eqref{eq:polyper}), there exists a sequence of $(P_n)$ admissible polygons for \eqref{eq:poly1classic} (resp. \eqref{eq:polyclassic} or  \eqref{eq:polyclassicper}) such that  
     $$
   \lim_{n\to+\infty} \lambda_{h,\beta}(P_n)=\overline{\lambda}_{h,\beta}(P) \quad\forall h\in\N,
    $$
    $$
    \lim_{n\to+\infty} F(\lambda_{1,\beta}(P_n),\ldots,\lambda_{k,\beta}(P_n))=F(\overline\lambda_{1,\beta}(P),\ldots,\overline\lambda_{k,\beta}(P)),
    $$
    $$ 
\lim_{n\to+\infty}|P_n|=|P|\quad\text{and}\quad  \lim_{n\to+\infty}\mathcal{H}^1(\partial P_n)=\widetilde{Per}(P).
    $$
In particular the infima of problems  \eqref{eq:poly1classic}, \eqref{eq:polyclassic} and \eqref{eq:polyclassicper} respectively coincide with the minima of problems \eqref{eq:poly1}, \eqref{eq:poly} and \eqref{eq:polyper}.
\end{proposition}
\begin{proof}
The proof is based on some simple observations. First of all, it is always possible to "detach the cracks", i.e. to split sides with multiplicity two and obtain two sides of multiplicity one; this new sides have  Hausdorff distance as small as we wish from the original crack. In fact, if the crack consists of only one segment $[A,B]$ of multiplicity two, two situations can occur.
If one vertex, say $B$, belongs to the interior of $\overline{P}$ then it is sufficient to take as a new endpoint a point $A_\varepsilon\in\partial P$ not lying on the same line as $[A,B]$ and having exactly distance $\varepsilon$ from $A$. The new polygon is obtained replacing one of the two versions of $[A,B]$ with $[A_\varepsilon,B]$. Notice that we increase neither the area (since we are taking a subset of $P$) nor the generalized perimeter (since in the triangle of vertices $A,B,A_\varepsilon$ we replace the two sides $[A,B]$ and $[A_\varepsilon,A]$ with the third side $[A_\varepsilon,B]$). If both vertices $A$ and $B$ do not belong to the interior of $\overline{P}$, we can still replace the side where $[A,B]$ lies without increasing either the volume or the generalized perimeter. In fact, if one of the intersecting sides is a subset of the other, it is sufficient to replace the shortest one with a segment $[A_\varepsilon,B_\varepsilon]$ obtained intersecting $P$ with a parallel line to $[A,B]$ at distance $\varepsilon$ from the original segment; otherwise, it is enough to move one vertex, $A$ or $B$, on another side it belongs and away from the intersection, giving a rotation around the fixed vertex that provides the required detachment. All the possible situations are shown in Figure 5.

\begin{figure}[h!]
\label{fig:pepsilon}
\begin{tikzpicture}[>=>>>]
\fill[green!50!gray!] (-2,0) .. controls (-2,0) and (2,0) ..(2,0) .. controls (2,0) and (2,3) ..(2,3) .. controls (2,3) and (0,2.5) .. (0,2.5) .. controls (0,2.5) and (0,2) .. (0,2) .. controls (0,2) and (1,2.1) .. (1,2.1) .. controls (1,2.1) and (1,1.5) .. (1,1.5) .. controls  (1,1.5) and (-0.5,1.5) .. (-0.5,1.5) .. controls (-0.5,1.5) and (-0.5,1) .. (-0.5,1) .. controls (-0.5,1) and (-1,1.5) .. (-1,1.5) .. controls (-1,1.5) and (0,1.5) .. (0,1.5).. controls (0,1.5) and (0,4) .. (0,4) .. controls (0,4) and (-2,0) .. (-2,0);
\draw[line width=.7pt]  (-2,0) .. controls (-2,0) and (2,0) ..(2,0) .. controls (2,0) and (2,3) ..(2,3) .. controls (2,3) and (0,2.5) .. (0,2.5) .. controls (0,2.5) and (0,2) .. (0,2) .. controls (0,2) and (1,2.1) .. (1,2.1) .. controls (1,2.1) and (1,1.5) .. (1,1.5) .. controls  (1,1.5)  and (-0.5,1.5) .. (-0.5,1.5) .. controls (-0.5,1.5) and (-0.5,1) .. (-0.5,1) .. controls (-0.5,1) and (-1,1.5) .. (-1,1.5) .. controls (-1,1.5) and (0,1.5) .. (0,1.5).. controls (0,1.5) and (0,4) .. (0,4) .. controls (0,4) and (-2,0) .. (-2,0);
\draw[line width=.7pt]  (-1,0) .. controls (-1,0) and (0,.5) ..(0,.5);
\draw (-2,2.3) node[anchor=north west] {$P$};
\fill[green!50!gray!20!] (-2+6,0) .. controls (-2+6,0) and (2+6,0) ..(2+6,0) .. controls (2+6,0) and (2+6,3) ..(2+6,3) .. controls (2+6,3) and (0+6,2.5) .. (0+6,2.5) .. controls (0+6,2.5) and (0+6,2) .. (0+6,2) .. controls (0+6,2) and (1+6,2.1) .. (1+6,2.1) .. controls (1+6,2.1) and (1+6,1.5) .. (1+6,1.5) .. controls  (1+6,1.5)
and (-0.5+6,1.5) .. (-0.5+6,1.5) .. controls (-0.5+6,1.5) and (-0.5+6,1) .. (-0.5+6,1) .. controls (-0.5+6,1)
and (-1+6,1.5) .. (-1+6,1.5) .. controls (-1+6,1.5) and (0+6,1.5) .. (0+6,1.5).. controls (0+6,1.5) and (0+6,4) .. (0+6,4) .. controls (0+6,4) and (-2+6,0) .. (-2+6,0);
\fill[green!50!gray!] (-2+6,0) .. controls (-2+6,0) and (-1+6,0) .. (-1+6,0) .. controls (-1+6,0) and (0+6,.5) .. (0+6,.5) .. controls (0+6,.5) and  (-1+6.3,0) .. (-1+6.3,0) .. controls (-1+6.3,0) and (2+6,0) ..(2+6,0) .. controls (2+6,0) and (2+6,3) ..(2+6,3) .. controls (2+6,3) and (0+6.15,2.5375) .. (0+6.15,2.5375) .. controls (0+6.15,2.5375) and (0+6.15,2.015) .. (0+6.15,2.015) .. controls (0+6.15,2.015) and (1+6,2.1) .. (1+6,2.1) .. controls (1+6,2.1) and (1+6,1.5) .. (1+6,1.5) .. controls  (1+6,1.5)  
and (-0.5+6,1.5) .. (-0.5+6,1.5) .. controls (-0.5+6,1.5) and (-0.5+6,1) .. (-0.5+6,1) .. controls (-0.5+6,1)
and (-1+6,1.5) .. (-1+6,1.5) .. controls (-1+6,1.5) and (0+6,1.65) .. (0+6,1.65).. controls (0+6,1.65) and (0+6,4) .. (0+6,4) .. controls (0+6,4) and (-2+6,0) .. (-2+6,0);
\draw[line width=.7pt]  (-2+6,0) .. controls (-2+6,0) and (2+6,0) ..(2+6,0) .. controls (2+6,0) and (2+6,3) ..(2+6,3) .. controls (2+6,3) and (0+6,2.5) .. (0+6,2.5) .. controls (0+6,2.5) and (0+6,2) .. (0+6,2) .. controls (0+6,2) and (1+6,2.1) .. (1+6,2.1) .. controls (1+6,2.1) and (1+6,1.5) .. (1+6,1.5) .. controls  (1+6,1.5)  
and (-0.5+6,1.5) .. (-0.5+6,1.5) .. controls (-0.5+6,1.5) and (-0.5+6,1) .. (-0.5+6,1) .. controls (-0.5+6,1)
and (-1+6,1.5) .. (-1+6,1.5) .. controls (-1+6,1.5) and (0+6,1.5) .. (0+6,1.5).. controls (0+6,1.5) and (0+6,4) .. (0+6,4) .. controls (0+6,4) and (-2+6,0) .. (-2+6,0);
\draw[line width=.7pt]  (-1+6,0) .. controls (-1+6,0) and (0+6,.5) ..(0+6,.5);
\draw[line width=.7pt, dashed]  (-1+6.3,0) .. controls (-1+6.3,0) and (0+6,.5) ..(0+6,.5);
\draw[line width=.7pt, dashed]  (0+6.15,2.5375) .. controls (0+6.15,2.5375) and (0+6.15,2.015) ..(0+6.15,2.015);
\draw[line width=.7pt, dashed]  (-1+6,1.5) .. controls (-1+6,1.5) and (0+6,1.65) .. (0+6,1.65);
\draw (-2+6,2.3) node[anchor=north] {$P\to P_\varepsilon$};
\fill[green!50!gray!] (-2+12,0) .. controls (-2+12,0) and (-1+12,0) .. (-1+12,0) .. controls (-1+12,0) and (0+12,.5) .. (0+12,.5) .. controls (0+12,.5) and  (-1+12.3,0) .. (-1+12.3,0) .. controls (-1+12.3,0) and (2+12,0) ..(2+12,0) .. controls (2+12,0) and (2+12,3) ..(2+12,3) .. controls (2+12,3) and (0+12.15,2.5375) .. (0+12.15,2.5375) .. controls (0+12.15,2.5375) and (0+12.15,2.015) .. (0+12.15,2.015) .. controls (0+12.15,2.015) and (1+12,2.1) .. (1+12,2.1) .. controls (1+12,2.1) and (1+12,1.5) .. (1+12,1.5) .. controls  (1+12,1.5)  
and (-0.5+12,1.5) .. (-0.5+12,1.5) .. controls (-0.5+12,1.5) and (-0.5+12,1) .. (-0.5+12,1) .. controls (-0.5+12,1)
and (-1+12,1.5) .. (-1+12,1.5) .. controls (-1+12,1.5) and (0+12,1.65) .. (0+12,1.65).. controls (0+12,1.65) and (0+12,4) .. (0+12,4) .. controls (0+12,4) and (-2+12,0) .. (-2+12,0);
\draw[line width=.7pt]  (-2+12,0) .. controls (-2+12,0) and (-1+12,0) .. (-1+12,0) .. controls (-1+12,0) and (0+12,.5) .. (0+12,.5) .. controls (0+12,.5) and  (-1+12.3,0) .. (-1+12.3,0) .. controls (-1+12.3,0) and (2+12,0) ..(2+12,0) .. controls (2+12,0) and (2+12,3) ..(2+12,3) .. controls (2+12,3) and (0+12.15,2.5375) .. (0+12.15,2.5375) .. controls (0+12.15,2.5375) and (0+12.15,2.015) .. (0+12.15,2.015) .. controls (0+12.15,2.015) and (1+12,2.1) .. (1+12,2.1) .. controls (1+12,2.1) and (1+12,1.5) .. (1+12,1.5) .. controls  (1+12,1.5) 
and (-0.5+12,1.5) .. (-0.5+12,1.5) .. controls (-0.5+12,1.5) and (-0.5+12,1) .. (-0.5+12,1) .. controls (-0.5+12,1)
and (-1+12,1.5) .. (-1+12,1.5) .. controls (-1+12,1.5) and (0+12,1.65) .. (0+12,1.65).. controls (0+12,1.65) and (0+12,4) .. (0+12,4) .. controls (0+12,4) and (-2+12,0) .. (-2+12,0);
\draw (-2+12,2.3) node[anchor=north west] {$P_\varepsilon$};
\end{tikzpicture}
\caption{Building $P_\varepsilon$ starting from $P$}
\end{figure}
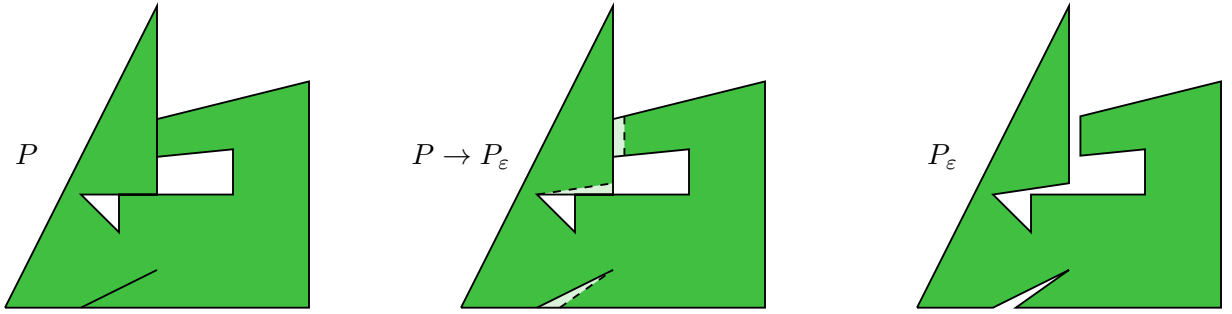

On the other hand, if a subset of $\partial P$ with multiplicity two consists of more consecutive segments, it is enough to apply the previous argument a finite number of times, starting from the external vertices and detaching one segment per time. Notice that in all cases, we do not increase the total number of sides, see Figure 6, because every step of the procedure involves the detachment of sides which had multiplicity \emph{two} creating \emph{two} disjoint sides. 

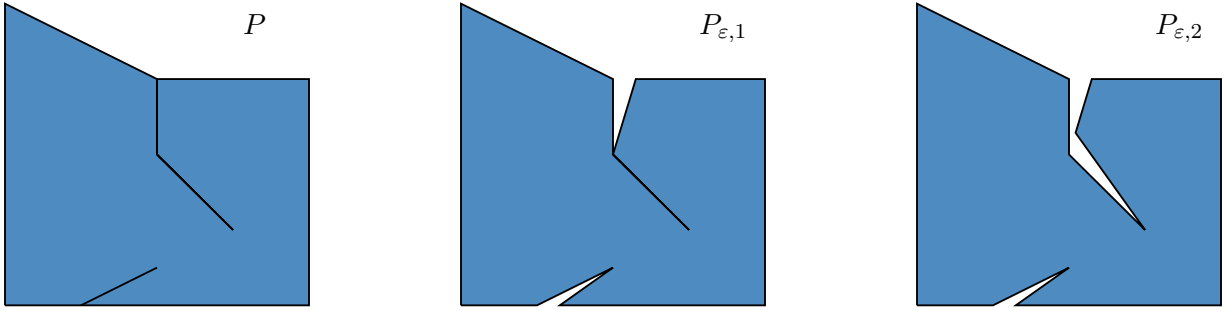
\begin{figure}[h!]
\begin{tikzpicture}[>=>>>]
\fill[blue!65!green!70!] (-2,0) .. controls (-2,0) and (2,0) ..(2,0) .. controls (2,0) and (2,3) ..(2,3) .. controls (2,3) and (0,3) .. (0,3) .. controls (0,3) and (0,2) .. (0,2) .. controls (0,2) and (1,1) .. (1,1) .. controls (1,1) and (0,2) .. (0,2) .. controls  (0,2)  and (0,3) .. (0,3) .. controls (0,3) and (-2,4) .. (-2,4).. controls (-2,4) and (-2,0) ..  (-2,0);
\draw[line width=.7pt] (-2,0) .. controls (-2,0) and (2,0) ..(2,0) .. controls (2,0) and (2,3) ..(2,3) .. controls (2,3) and (0,3) .. (0,3) .. controls (0,3) and (0,2) .. (0,2) .. controls (0,2) and (1,1) .. (1,1) .. controls (1,1) and (0,2) .. (0,2) .. controls  (0,2)  and (0,3) .. (0,3) .. controls (0,3) and (-2,4) .. (-2,4).. controls (-2,4) and (-2,0) ..  (-2,0);
\draw[line width=.7pt]  (-1,0) .. controls (-1,0) and (0,.5) ..(0,.5);
\draw (1,4) node[anchor=north west] {$P$};
\fill[blue!65!green!70!] (-2+6,0) .. controls (-2+6,0) and (-1+6,0) .. (-1+6,0) .. controls (-1+6,0) and (0+6,.5) .. (0+6,.5) .. controls (0+6,.5) and  (-1+6.3,0) .. (-1+6.3,0) .. controls (-1+6.3,0) and (2+6,0) ..(2+6,0) .. controls (2+6,0) and (2+6,3) ..(2+6,3) .. controls (2+6,3) and (0+6.3,3) .. (0+6.3,3) .. controls (0+6.3,3) and (0+6,2) .. (0+6,2) .. controls (0+6,2) and (1+6,1) .. (1+6,1) .. controls (1+6,1) and (0+6,2) .. (0+6,2) .. controls  (0+6,2)  and (0+6,3) .. (0+6,3) .. controls (0+6,3) and (-2+6,4) .. (-2+6,4).. controls (-2+6,4) and (-2+6,0) .. (-2+6,0);
\draw[line width=.7pt]  (-2+6,0) .. controls (-2+6,0) and (-1+6,0) .. (-1+6,0) .. controls (-1+6,0) and (0+6,.5) .. (0+6,.5) .. controls (0+6,.5) and  (-1+6.3,0) .. (-1+6.3,0) .. controls (-1+6.3,0) and (2+6,0) ..(2+6,0) .. controls (2+6,0) and (2+6,3) ..(2+6,3) .. controls (2+6,3) and (0+6.3,3) .. (0+6.3,3) .. controls (0+6.3,3) and (0+6,2) .. (0+6,2) .. controls (0+6,2) and (1+6,1) .. (1+6,1) .. controls (1+6,1) and (0+6,2) .. (0+6,2) .. controls  (0+6,2)  and (0+6,3) .. (0+6,3) .. controls (0+6,3) and (-2+6,4) .. (-2+6,4).. controls (-2+6,4) and (-2+6,0) .. (-2+6,0);
\draw (1+6,4) node[anchor=north west] {$P_{\varepsilon,1}$};
\fill[blue!65!green!70!] (-2+12,0) .. controls (-2+12,0) and (-1+12,0) .. (-1+12,0) .. controls (-1+12,0) and (0+12,.5) .. (0+12,.5) .. controls (0+12,.5) and  (-1+12.3,0) .. (-1+12.3,0) .. controls (-1+12.3,0) and (2+12,0) ..(2+12,0) .. controls (2+12,0) and (2+12,3) ..(2+12,3) .. controls (2+12,3) and (0+12.3,3) .. (0+12.3,3) .. controls (0+12.3,3) and (12.0862,2.2873) .. (12.0862,2.2873) .. controls (12.0862,2.2873) and (1+12,1) .. (1+12,1) .. controls (1+12,1) and (0+12,2) .. (0+12,2) .. controls  (0+12,2)  and (0+12,3) .. (0+12,3) .. controls (0+12,3) and (-2+12,4) .. (-2+12,4).. controls (-2+12,4) and (-2+12,0) .. (-2+12,0);
\draw[line width=.7pt]  (-2+12,0) .. controls (-2+12,0) and (-1+12,0) .. (-1+12,0) .. controls (-1+12,0) and (0+12,.5) .. (0+12,.5) .. controls (0+12,.5) and  (-1+12.3,0) .. (-1+12.3,0) .. controls (-1+12.3,0) and (2+12,0) ..(2+12,0) .. controls (2+12,0) and (2+12,3) ..(2+12,3) .. controls (2+12,3) and (0+12.3,3) .. (0+12.3,3) .. controls (0+12.3,3) and (12.0862,2.2873) .. (12.0862,2.2873) .. controls (12.0862,2.2873) and (1+12,1) .. (1+12,1) .. controls (1+12,1) and (0+12,2) .. (0+12,2) .. controls  (0+12,2)  and (0+12,3) .. (0+12,3) .. controls (0+12,3) and (-2+12,4) .. (-2+12,4).. controls (-2+12,4) and (-2+12,0) .. (-2+12,0);
\draw (1+12,4) node[anchor=north west] {$P_{\varepsilon,2}$};
\end{tikzpicture}
\caption{When a crack of $P$ consists of more consecutive segments, we apply the procedure several times}
\end{figure}

Considering now $\varepsilon=1/n$, we build a sequence of unions of simple polygons with at most $N$ sides such that
$$
P_n\xrightarrow{H^c}P,\quad \lim_{n\to+\infty}|P_n|=|P|\quad\text{and}\quad  \lim_{n\to+\infty}\mathcal{H}^1(\partial P_n)=\widetilde{Per}(P);
$$
The claim is thus an immediate consequence of the continuity of the eigenvalues proved in \cite[Theorem 6.1, Theorem 6.2]{bugitr22}. 
\end{proof}

\section{Negative boundary parameter}\label{sec:negative}

In this section we consider the case $\beta<0$ and set $\eta:=-\beta>0$. We deal with the problem
 \begin{equation}\label{eq:negativeRobin}
        \left\{\begin{array}{ll}
        \Delta u + \lambda u = 0 \qquad &\text{in\ \ } \Omega
        \\
        \displaystyle   \frac{\partial u}{\partial\nu}-\eta u=0 & \text{on\ \ }  \partial\Omega.
        \end{array}\right.
    \end{equation}
The existence result follows the same scheme as in \cite{bucurcito}. Due to the pathologic behaviour of the eigenvalues of shrinking sets, we are not able to discuss the optimization problem in the same generality as for $\beta>0$.
\begin{remark}\label{rem:constraint}
A reasonable question is why we do not consider the more general problems
\begin{equation}\label{eq:polynegvol}
\max\left\{F(\overline{\lambda}_{1,-\eta}(P),\ldots,\overline{\lambda}_{k,-\eta}(P)): P\in\overline{\mathcal{P}_{N}},\ |P|\le m\right\},
\end{equation}
\begin{equation}\label{eq:polynegperin}
\max\left\{F(\overline{\lambda}_{1,-\eta}(P),\ldots,\overline{\lambda}_{k,-\eta}(P)): P\in\overline{\mathcal{P}_{N}},\ \widetilde{Per}(P)\le p\right\},
\end{equation}
with $F$ satisfying the same hypotheses as in Problem \eqref{eq:polyneg}.
Let us consider a simple polygon $P\in\mathcal{P}_N$ of measure $m$ and let us consider the family $tP$, which is decreasing as $t\downarrow 0$. By \eqref{negativebeta} we have 
$$
\overline{\lambda}_{1,-\eta}(tP)={\lambda}_{1,-\eta}(tP)\le-\eta\frac{t\mathcal{H}^1(\partial P)}{t^2|P|}\le-\eta\frac{ 2\sqrt{\pi}}{t\sqrt{m}}\xrightarrow{t\to 0^+}-\infty.
$$
In particular, we can deduce the asymptotic behaviour of ${\lambda}_{1,-\eta}(tP)$ using \eqref{eq:4.16} and De L'H\^{o}pital's rule; indeed one has
$$
\lim_{t\to 0^+}\frac{\lambda_{1,-\eta}(tP)}{\frac{1}{t}}=\lim_{t\to 0^+}\frac{\frac{1}{t^2}{\lambda}_{1,-t\eta}(P)}{\frac{1}{t}}=\lim_{t\to 0^+}\frac{\lambda_{1,-t\eta}(P)}{t}=\lim_{t\to 0^+}\frac{\frac{d}{dt}\lambda_{1,-t\eta}(P)}{1}=-\eta\frac{\mathcal{H}^1(\partial P)}{|P|}.
$$
So $\lambda_{1,-\eta}(tP)\simeq -\eta\frac{\mathcal{H}^1(\partial P)}{t|P|}$ as $t\to 0^+$.
Now let us consider a nonequilateral triangle $T\in\mathcal{P}_N$ and use \eqref{eq:expansion}:
$$
\lim_{t\to 0^+}\frac{\lambda_{2,-\eta}(tT)}{\frac{1}{t^2}}=\lim_{t\to 0^+}\frac{\frac{1}{t^2}{\lambda}_{2,-t\eta}(T)}{\frac{1}{t^2}}=\lim_{t\to 0^+}\lambda_{2,-t\eta}(T)=\mu_2(T)>0.
$$
So $\lambda_{2,-\eta}(tT)\simeq \frac{\mu_2(T)}{t^2}$ as $t\to 0^+$. 
Taking
$$F(\xi_1,\xi_2):=\xi_1+\xi_2$$
we have that the family $tT$ is admissible for both Problems \eqref{eq:polynegvol} and \eqref{eq:polynegperin} for every $N\ge 3$ and it holds
$$\lim_{t\to 0^+}F(\lambda_{1,-\eta},(tT),\lambda_{2,-\eta}(tT))=\lim_{t\to 0^+}\left(-\eta\frac{\mathcal{H}^1(\partial P)}{t|P|}+\frac{\mu_2(T)}{t^2}\right)=+\infty.$$
So Problems \eqref{eq:polynegvol} and \eqref{eq:polynegperin} do not admit maximum in general.
\end{remark}

We now show that if $\overline{\lambda}_{k,-\eta}(P)$ is not too small then the number of the connected components of $P$ and their diameter are bounded according to the size of $\overline{\lambda}_{k,\beta}(P)$, see \cite[Proposition 14]{bucurcito}. 

\begin{proposition}[A priori bound on diameter and on the number of connected components]\label{pro:isod}
Let $N\in\N$, $P\in\overline{\mathcal{P}_N}$, $|P|=m$ and let $A>0$ be such that $\overline{\lambda}_{k,\beta}(P)>-A$. Then
$P$ is union of $M$ equibounded connected components
$$
P=P_1\cup\ldots\cup P_M,
$$
with $M<\frac{m A^2}{4\pi\eta^2}+k$ and ${\rm diam}(P_j)\le D(m,\beta,k,A)$, i.e., the diameters of the connected components are uniformly bounded.
\end{proposition}

We are now in a position to prove the main result of the section.

\begin{theorem}[Existence of a maximal generalized polygon]\label{teo:exis}
Problems \eqref{eq:polyneg} and \eqref{eq:polynegper} admit solutions in $\overline{{\mathcal P}_N}$. Each optimal polygon $P$ is bounded and can be written as the union of at most 
$$
\min\left\{\left\lfloor\frac{N-1}{2}\right\rfloor,\frac{mA_*^2}{4\pi\eta^2}+k\right\}
$$
equibounded connected components, where $A_*>0$ is such that
$$
F(\lambda_{1,-\eta}(E_N),\ldots,\lambda_{k,-\eta}(E_N))>F(-A_*,\ldots,-A_*)
$$
with $E_N$ the regular $N$-agon saturating the constraint of the problem.
\end{theorem}
\begin{proof}
Let $(P_n)\subset\overline{{\mathcal P}_N}$ be a maximizing sequence for $F(\overline{\lambda}_{1,-\eta}(\cdot),\ldots,\overline{\lambda}_{k,-\eta}(\cdot))$. In view of the hypotheses on $F$, we observe that any admissible polygon $E$ such that $\overline{\lambda}_{h,-\eta}(E)\le-A_*$ for every $h=1,\ldots,k$ cannot be optimal. Then, it is not restrictive to assume that $\overline{\lambda}_{h,-\eta}(P_n)>-A_*$ for every $h=1,\ldots,k$. By Proposition \ref{pro:isod} we have that $\text{diam}(P_n)<D$ for some $D$ independent of $n$. As a consequence, since we are in the polygonal framework, we also have $\sup_{n\in\N}\mathcal{H}^1(\partial P_n)<+\infty$. Thanks again to Proposition \ref{pro:isod}, we can write $P_n$ as union of $M_n$ equibounded connected components $P_n^1,\ldots,P_n^{M_n}$ as follows:
$$
P_n=P_n^1\cup\ldots\cup P_n^{M_n}, \quad M_n\le\min\left\{\left\lfloor\frac{N-1}{2}\right\rfloor,\frac{mA_*^2}{4\pi\eta^2}+k\right\}.
$$
The bound on $M_n$ is given both by Proposition \ref{pro:isod} and by the fact that we are dealing with polygons with at most $N$ sides, see Remark \ref{pro:compoly}(iii). These facts entail the existence of $P\in\overline{\mathcal{P}_N}$ such that $P_n\xrightarrow{H^c} P$ (up to subsequences).

Now, it remains to prove that $F(\overline{\lambda}_{1,-\eta}(\cdot),\ldots,\overline{\lambda}_{k,-\eta}(\cdot))$ is upper semicontinuous in $\overline{{\mathcal P}_N}$ with respect to the $H^c$-convergence. This has already been proved in \cite[Proposition 18]{bucurcito}; we report the main points of the proof for the convenience of the reader.

Let us show that, for any $h=1,\ldots,k$, we have
\begin{equation}
\label{eq:upsemicon}
\overline{\lambda}_{h,-\eta}(P)\ge\limsup_{n\to+\infty}\overline{\lambda}_{h,-\eta}(P_n);
\end{equation}
the upper semicontinuity of $F$ in each variable will give the thesis. Let us fix $\varepsilon>0$ and let $S$ be an admissible vector space in the min-max formula \eqref{eq:genEigenvalues} for $\overline{\lambda}_{h,-\eta}(P)$ such that
\begin{equation}\label{eq:lsc2d1}
\overline{\lambda}_{h,-\eta}(P)\ge\max_{u\in S\setminus\left\{0\right\}} \overline{R}_{P,-\eta}(u)
-\varepsilon.
\end{equation}
Let $\left\{u_j:j=1,\ldots,h\right\}$ be an $L^2(P)$-orthonormal basis for $S$. Then, for every $j=1,\ldots, h$, there exists $v_j^n\in H^1(P_n)$ such that, denoting by the same symbol the extension by zero of a function outside its domain, $v_j^n\to u_j$ strongly in $L^2(\R^2)$ and $\nabla v_j^n\to \nabla u_j$ strongly in $L^2(\R^2;\R^2)$. Now, since $\left\{u_1,\ldots,u_h\right\}$ is $L^2(P)$-orthonormal and $P_n\to P$ in $L^1(\R^2)$, we deduce that, for $n\in\N$ sufficiently large, $\left\{v_1^n,\ldots,v_h^n\right\}$ can be chosen linearly independent in $L^2(P_n)$. Let $S_n:=\text{span}\left\{v_1^n,\ldots,v_h^n\right\}$; it is an admissible subspace for the computation of $\overline{\lambda}_{h,-\eta}(P_n)$. Let
$$
v^n=\sum_{j=1}^h\alpha_j^nv_j^n\in S_n
$$
realize the maximum for the generalized Rayleigh quotient $\overline{R}$ on $S_n$:
$$
\max_{w\in S_n}\overline{R}(w)=\overline{R}(v^n).
$$
Without loss of generality, we can assume
$\sum_{j=1}^h(\alpha_j^n)^2=1$.
Then, up to subsequences, $\alpha_j^n\to\alpha_j$ in $\R$, with $\sum_{j=1}^h(\alpha_j)^2=1$.
Setting
$$
v=\sum_{j=1}^h\alpha_ju_j,
$$
we have that $v\in S\setminus\left\{0\right\}$, $v^n\to v$ strongly in $L^2(\R^2)$ and $\nabla v^n\to \nabla v$ in $L^2(\R^2;\R^2)$. Using \eqref{eq:lsc2d1}, the continuity of the volume integrals and the lower semicontinuity of the boundary integral (see Lemma \ref{teo:lscboundary}), we have
\begin{align*}
\limsup_{n\to+\infty}\overline{\lambda}_{h,-\eta}(P_n)&\le\limsup_{n\to+\infty}\sup_{w\in S_n}\overline{R}_{P_n,-\eta}(w) \le\limsup_{n\to+\infty}\overline{R}_{P,-\eta}(v^n)+\varepsilon\\
&\le\overline{R}_{P,-\eta}(v)+\varepsilon\le\max_{u\in S\setminus\left\{0\right\}}\overline{R}_{P,-\eta}(u)\le\overline{\lambda}_{h,-\eta}(P)+\varepsilon.
\end{align*}
Letting $\varepsilon\to 0^+$ we obtain \eqref{eq:upsemicon}; this concludes the proof.
\end{proof}

When we deal with the negative boundary parameter case, it has been seen in several situations that the perimeter constraint turns out to be rather natural; see, for instance,  \cite{bfnt19,cito2021quantitative} where the optimality and the stability of the ball for the first eigenvalue in the convex case is addressed. Even in our framework, the (generalized) perimeter constraint gives some additional properties of the optimal shapes. The following result gives a further property of the solutions of \eqref{eq:polynegper}, whenever all the eigenvalues involved in the functional are negative. This happens, for instance, if
$$
F(x_1,\ldots,x_k)=x_1
$$
and the $k$-th eigenvalue of an optimal polygon $P$ is negative. The condition $\overline{\lambda}_{k,-\eta}(P)<0$ is satisfied if $\eta$ is sufficiently large, more precisely if $\eta>\overline{\sigma}_k(P)$, were $\overline{\sigma}_k(P)$ is the $k$-th generalized Steklov eigenvalue of $P$, see e.g. \cite{BucFreKen}. Notice that $\overline{\sigma}_1(P)=0$ with eigenfunction given by the characteristic function of $P$, and so the strict negativity of the first Robin eigenvalue for all $\eta>0=\overline{\sigma}_1(P)$ is coherent with the previous consideration.

\begin{proposition}\label{Prop:furtherproperties}
For every polygon $P\in\overline{\mathcal{P}_{N}}$ with $\overline{\lambda}_{k,-\eta}(P)<0$, there exists a polygon $P^\bullet\in\overline{\mathcal{P}_{N}}$ which is a union of simple polygons, such that $\widetilde{Per}(P^\bullet)\le \widetilde{Per}(P)$ and  
\begin{equation}\label{eq:bullet}
    \overline{\lambda}_{h,-\eta}(P^\bullet)\ge\overline{\lambda}_{h,-\eta}(P).
\end{equation}
In particular, every solution $P_0$ of
Problem \eqref{eq:polynegper}
is union of simple polygons and thus $\overline{\lambda}_{1,-\eta}(P_0)={\lambda}_{1,-\eta}(P_0)$.
\end{proposition}
\begin{proof}
Let $P$ be as in the statement. Let us denote by $U$ the unbounded connected component of $\R^2\setminus\overline{P}$ and consider
$$
P^\bullet:=\R^2\setminus\overline{U}.
$$
The set $P^\bullet$ is obtained from $P$ by filling the holes and the fractures, see Figure \ref{fig:bullet}. This construction makes $P^\bullet$ a finite union of simple polygons; it is clear that $P^\bullet$ has at most $N$ sides and $\widetilde{Per}(P^\bullet) \leq \widetilde{Per}(P)$.
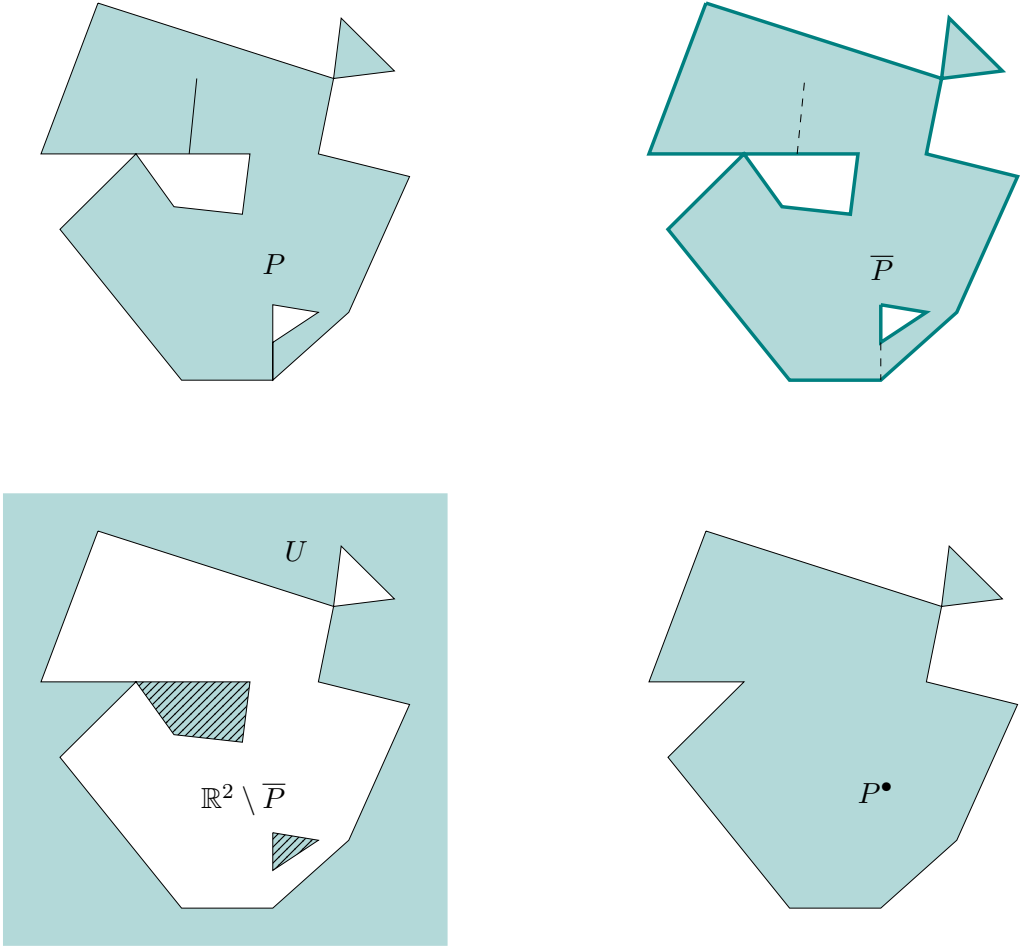
\begin{figure}[h!]
\begin{tikzpicture}[>=>>>]
\fill[teal!30] (-1,5) .. controls (-1,5) and (-1.75,3) ..(-1.75,3) .. controls (-1.75,3) and (1,3) .. (1,3) .. controls (1,3) and (0.9,2.2) .. (0.9,2.2) .. controls (0.9,2.2) and (0,2.3) .. (0,2.3) .. controls (0,2.3) and (-.5,3) .. (-.5,3) .. controls (-.5,3) and (-1.5,2) .. (-1.5,2) .. controls (-1.5,2) and (.1,0) .. (.1,0) .. controls (.1,0) and (1.3,0) .. (1.3,0) .. controls (1.3,0)
and (1.3,1) .. (1.3,1) .. controls (1.3,1) and (1.9,.9) .. (1.9,.9) .. controls (1.9,.9) and (1.3,.5) .. (1.3,.5) .. controls (1.3,.5) and (1.3,0) .. (1.3,0) .. controls (1.3,0) and (2.3,.9) .. (2.3,.9).. controls (2.3,.9) and (3.1,2.7) .. (3.1,2.7).. controls (3.1,2.7) and (1.9,3) .. (1.9,3).. controls (1.9,3) and (2.1,4) .. (2.1,4) .. controls (2.1,4) and (2.9,4.1) .. (2.9,4.1).. controls (2.9,4.1) and (2.2,4.8) .. (2.2,4.8).. controls (2.2,4.8) and (2.1,4) .. (2.1,4);
\draw[line width=.2pt](-1,5) .. controls (-1,5) and (-1.75,3) ..(-1.75,3) .. controls (-1.75,3) and (1,3) .. (1,3) .. controls (1,3) and (0.9,2.2) .. (0.9,2.2) .. controls (0.9,2.2) and (0,2.3) .. (0,2.3) .. controls (0,2.3) and (-.5,3) .. (-.5,3) .. controls (-.5,3) and (-1.5,2) .. (-1.5,2) .. controls (-1.5,2) and (.1,0) .. (.1,0) .. controls (.1,0) and (1.3,0) .. (1.3,0) .. controls (1.3,0) and (1.3,1) .. (1.3,1) .. controls (1.3,1) and (1.9,.9) .. (1.9,.9) .. controls (1.9,.9) and (1.3,.5) .. (1.3,.5) .. controls (1.3,.5) and (1.3,0) .. (1.3,0) .. controls (1.3,0) and (2.3,.9) .. (2.3,.9).. controls (2.3,.9) and (3.1,2.7) .. (3.1,2.7).. controls (3.1,2.7) and (1.9,3) .. (1.9,3).. controls (1.9,3) and (2.1,4) .. (2.1,4) .. controls (2.1,4) and (2.9,4.1) .. (2.9,4.1).. controls (2.9,4.1) and (2.2,4.8) .. (2.2,4.8).. controls (2.2,4.8) and (2.1,4) .. (2.1,4) .. controls (2.1,4) and (-1,5) .. (-1,5);
\draw[line width=.2pt] (.2,3) .. controls (.2,3) and (.3,4) .. (.3,4);
\draw (1.6,1.8) node[anchor=north east] {$P$};
\fill[teal!30] (+8-1,5) .. controls (+8-1,5) and (+8-1.75,3) ..(+8-1.75,3) .. controls (+8-1.75,3) and (1+8,3) .. (1+8,3) .. controls (1+8,3) and (0.9+8,2.2) .. (0.9+8,2.2) .. controls (0.9+8,2.2) and (0+8,2.3) .. (0+8,2.3) .. controls (0+8,2.3) and (-.5+8,3) .. (-.5+8,3) .. controls (-.5+8,3) and (-1.5+8,2) .. (-1.5+8,2) .. controls (-1.5+8,2) and (.1+8,0) .. (.1+8,0) .. controls (.1+8,0) and (1.3+8,0) .. (1.3+8,0) .. controls (1.3+8,0)
and (1.3+8,1) .. (1.3+8,1) .. controls (1.3+8,1) and (1.9+8,.9) .. (1.9+8,.9) .. controls (1.9+8,.9) and (1.3+8,.5) .. (1.3+8,.5) .. controls (1.3+8,.5) and (1.3+8,0) .. (1.3+8,0) .. controls (1.3+8,0) and (2.3+8,.9) .. (2.3+8,.9).. controls (2.3+8,.9) and (3.1+8,2.7) .. (3.1+8,2.7).. controls (3.1+8,2.7) and (1.9+8,3) .. (1.9+8,3).. controls (1.9+8,3) and (2.1+8,4) .. (2.1+8,4)  .. controls (2.1+8,4) and (2.9+8,4.1) .. (2.9+8,4.1).. controls (2.9+8,4.1) and (2.2+8,4.8) .. (2.2+8,4.8).. controls (2.2+8,4.8) and (2.1+8,4) .. (2.1+8,4);
\draw[line width=1.3pt, teal](-1+8,5) .. controls (-1+8,5) and (-1.75+8,3) ..(-1.75+8,3) .. controls (-1.75+8,3) and (1+8,3) .. (1+8,3) .. controls (1+8,3) and (0.9+8,2.2) .. (0.9+8,2.2) .. controls (0.9+8,2.2) and (0+8,2.3) .. (0+8,2.3) .. controls (0+8,2.3) and (-.5+8,3) .. (-.5+8,3) .. controls (-.5+8,3) and (-1.5+8,2) .. (-1.5+8,2) .. controls (-1.5+8,2) and (.1+8,0) .. (.1+8,0) .. controls (.1+8,0) and (1.3+8,0) .. (1.3+8,0) .. controls (1.3+8,0)  and (2.3+8,.9) .. (2.3+8,.9).. controls (2.3+8,.9) and (3.1+8,2.7) .. (3.1+8,2.7).. controls (3.1+8,2.7) and (1.9+8,3) .. (1.9+8,3).. controls (1.9+8,3) and (2.1+8,4) .. (2.1+8,4)  .. controls (2.1+8,4) and (2.9+8,4.1) .. (2.9+8,4.1).. controls (2.9+8,4.1) and (2.2+8,4.8) .. (2.2+8,4.8).. controls (2.2+8,4.8) and (2.1+8,4) .. (2.1+8,4) .. controls (2.1+8,4) and (-1+8,5) .. (-1+8,5);
\draw[line width=1.3pt, teal] (1.3+8,1) .. controls (1.3+8,1) and (1.9+8,.9) .. (1.9+8,.9) .. controls (1.9+8,.9) and (1.3+8,.5) .. (1.3+8,.5) .. controls (1.3+8,.5) and (1.3+8,1) ..(1.3+8,1);
\draw[line width=.2pt, dashed] (1.3+8,.5) .. controls (1.3+8,.5) and (1.3+8,0) .. (1.3+8,0);
\draw[line width=.2pt, dashed] (.2+8,3) .. controls (.2+8,3) and (.3+8,4) .. (.3+8,4);
\draw (1.6+8,1.8) node[anchor=north east] {$\overline{P}$};
\fill[teal!30] 
(-2.25,-1.5) .. controls (-2.25,-1.5) and (-1,5-7) ..
(-1,5-7) .. controls (-1,5-7) and (-1.75,3-7) ..(-1.75,3-7) .. controls (-1.75,3-7) and (1,3-7) .. (1,3-7) .. controls (1,3-7) and (0.9,2.2-7) .. (0.9,2.2-7) .. controls (0.9,2.2-7) and (0,2.3-7) .. (0,2.3-7) .. controls (0,2.3-7) and (-.5,3-7) .. (-.5,3-7) .. controls (-.5,3-7) and (-1.5,2-7) .. (-1.5,2-7) .. controls (-1.5,2-7) and (.1,0-7) .. (.1,0-7) .. controls (.1,0-7) and (1.3,0-7) .. (1.3,0-7) .. controls (1.3,0-7) and (1.3,1-7) .. (1.3,1-7) .. controls (1.3,1-7) and (1.9,.9-7) .. (1.9,.9-7) .. controls (1.9,.9-7) and (1.3,.5-7) .. (1.3,.5-7) .. controls (1.3,.5-7) and (1.3,0-7) .. (1.3,0-7) .. controls (1.3,0-7) and (2.3,.9-7) .. (2.3,.9-7).. controls (2.3,.9-7) and (3.1,2.7-7) .. (3.1,2.7-7).. controls (3.1,2.7-7) and (1.9,3-7) .. (1.9,3-7).. controls (1.9,3-7) and (2.1,4-7) .. (2.1,4-7) .. controls (2.1,4-7) and (2.9,4.1-7) .. (2.9,4.1-7).. controls (2.9,4.1-7) and (2.2,4.8-7) .. (2.2,4.8-7).. controls (2.2,4.8-7) and (2.1,4-7) .. (2.1,4-7) .. controls (2.1,4-7) and (-1,5-7) .. (-1,5-7)
.. controls (-1,5-7) and (-2.25,-1.5) .. (-2.25,-1.5) .. controls (-2.25,-1.5) and (3.6,-1.5) .. (3.6,-1.5) .. controls (3.6,-1.5) and (3.6,-7.5) .. (3.6,-7.5) .. controls (3.6,-7.5) and (-2.25,-7.5) .. (-2.25,-7.5) .. controls (-2.25,-7.5) and (-2.25,-1.5) .. (-2.25,-1.5);
\fill[pattern=north east lines]
(1.3,1-7) .. controls (1.3,1-7) and (1.9,.9-7) .. (1.9,.9-7) .. controls (1.9,.9-7) and (1.3,.5-7) .. (1.3,.5-7);
\fill[pattern=north east lines]
(1,3-7) .. controls (1,3-7) and (0.9,2.2-7) .. (0.9,2.2-7) .. controls (0.9,2.2-7) and (0,2.3-7) .. (0,2.3-7) .. controls (0,2.3-7) and (-.5,3-7) .. (-.5,3-7);
\draw[line width=0.2pt](-1,5-7) .. controls (-1,5-7) and (-1.75,3-7) ..(-1.75,3-7) .. controls (-1.75,3-7) and (1,3-7) .. (1,3-7) .. controls (1,3-7) and (0.9,2.2-7) .. (0.9,2.2-7) .. controls (0.9,2.2-7) and (0,2.3-7) .. (0,2.3-7) .. controls (0,2.3-7) and (-.5,3-7) .. (-.5,3-7) .. controls (-.5,3-7) and (-1.5,2-7) .. (-1.5,2-7) .. controls (-1.5,2-7) and (.1,0-7) .. (.1,0-7) .. controls (.1,0-7) and (1.3,0-7) .. (1.3,0-7) .. controls (1.3,0-7) and (2.3,.9-7) .. (2.3,.9-7).. controls (2.3,.9-7) and (3.1,2.7-7) .. (3.1,2.7-7).. controls (3.1,2.7-7) and (1.9,3-7) .. (1.9,3-7).. controls (1.9,3-7) and (2.1,4-7) .. (2.1,4-7) .. controls (2.1,4-7) and (2.9,4.1-7) .. (2.9,4.1-7).. controls (2.9,4.1-7) and (2.2,4.8-7) .. (2.2,4.8-7).. controls (2.2,4.8-7) and (2.1,4-7) .. (2.1,4-7) .. controls (2.1,4-7) and (-1,5-7) .. (-1,5-7);
\draw[line width=0.2pt] (1.3,1-7) .. controls (1.3,1-7) and (1.9,.9-7) .. (1.9,.9-7) .. controls (1.9,.9-7) and (1.3,.5-7) .. (1.3,.5-7) .. controls (1.3,.5-7) and (1.3,1-7) .. (1.3,1-7);
\draw (1.6,1.8-7) node[anchor=north east] {$\R^2\setminus\overline{P}$};
\draw (1.9,-2) node[anchor=north east]{$U$};
\fill[teal!30] (+8-1,5-7) .. controls (+8-1,5-7) and (+8-1.75,3-7) ..(+8-1.75,3-7) .. controls (+8-1.75,3-7) and (-.5+8,3-7) .. (-.5+8,3-7) .. controls (-.5+8,3-7) and (-1.5+8,2-7) .. (-1.5+8,2-7) .. controls (-1.5+8,2-7) and (.1+8,0-7) .. (.1+8,0-7) .. controls (.1+8,0-7) and (1.3+8,0-7) .. (1.3+8,0-7) .. controls (1.3+8,0-7) and (2.3+8,.9-7) .. (2.3+8,.9-7).. controls (2.3+8,.9-7) and (3.1+8,2.7-7) .. (3.1+8,2.7-7).. controls (3.1+8,2.7-7) and (1.9+8,3-7) .. (1.9+8,3-7).. controls (1.9+8,3-7) and (2.1+8,4-7) .. (2.1+8,4-7) .. controls (2.1+8,4-7) and (2.9+8,4.1-7) .. (2.9+8,4.1-7).. controls (2.9+8,4.1-7) and (2.2+8,4.8-7) .. (2.2+8,4.8-7).. controls (2.2+8,4.8-7) and (2.1+8,4-7) .. (2.1+8,4-7);
\draw[line width=.2pt](-1+8,5-7) .. controls (-1+8,5-7) and (-1.75+8,3-7) ..(-1.75+8,3-7) .. controls (-1.75+8,3-7) and (-.5+8,3-7) .. (-.5+8,3-7) .. controls (-.5+8,3-7) and (-1.5+8,2-7) .. (-1.5+8,2-7) .. controls (-1.5+8,2-7) and (.1+8,0-7) .. (.1+8,0-7) .. controls (.1+8,0-7) and (1.3+8,0-7) .. (1.3+8,0-7) .. controls (1.3+8,0-7)  and (2.3+8,.9-7) .. (2.3+8,.9-7).. controls (2.3+8,.9-7) and (3.1+8,2.7-7) .. (3.1+8,2.7-7).. controls (3.1+8,2.7-7) and (1.9+8,3-7) .. (1.9+8,3-7).. controls (1.9+8,3-7) and (2.1+8,4-7) .. (2.1+8,4-7) .. controls (2.1+8,4-7) and (2.9+8,4.1-7) .. (2.9+8,4.1-7).. controls (2.9+8,4.1-7) and (2.2+8,4.8-7) .. (2.2+8,4.8-7).. controls (2.2+8,4.8-7) and (2.1+8,4-7) .. (2.1+8,4-7) .. controls (2.1+8,4-7) and (-1+8,5-7) .. (-1+8,5-7);
\draw (1.6+8,1.8-7) node[anchor=north east] {$P^\bullet$};
\end{tikzpicture}
\label{fig:bullet}
\caption{Construction of $P^\bullet$}
\end{figure}
If $\overline{\lambda}_{h,-\eta}(P^\bullet)\ge 0$, then \eqref{eq:bullet} is immediate since $\overline{\lambda}_{h,-\eta}(P)<0$. Let us assume now that $\overline{\lambda}_{h,-\eta}(P^\bullet)<0$. Now, we fix $\varepsilon<|\overline{\lambda}_{h,-\eta}(P^\bullet)|$ and consider an $h$-dimensional subspace $S=\text{span}\left\{u_1,\ldots,u_h\right\}$ of $H^1(P^\bullet)$ such that
$$
\overline{\lambda}_{h,-\eta}(P^\bullet)+\varepsilon\ge\max_{\alpha_1,\ldots,\alpha_h\in\R}\overline{R}_{P^\bullet,-\eta}.
$$
We claim that the space generated by the restrictions of $u_1,\ldots,u_h$ to $P$, still denoted by $S$, is also an $h$-dimensional subspace of $H^1(P)$, so it is admissible to compute $\overline{\lambda}_{h,-\eta}(P)$. Indeed, let us assume, by contradiction, that the restrictions of $u_1,\ldots,u_h$ to $P$ are linearly dependent, so there exist $\tilde{\alpha}_1,\ldots,\tilde{\alpha}_h\in\R$ an $h$-tuple of coefficients such that
$$\sum_{j=1}^h\tilde{\alpha}_ju_j=0\quad\text{on $P$}.$$
It follows that both traces on $\partial P$ are zero, as well. Let us denote now by
$$\tilde{v}:=\sum_{j=1}^h\tilde{\alpha}_ju_j\in S.$$
Then, in view of the linear independence in $P^\bullet$, $\tilde{v}\not\equiv 0$ on $P^\bullet\setminus P$ and it holds $\tilde{v}=0$ on $\partial P^\bullet$ (in the sense of the traces). As a consequence we get
$$0<\frac{\int_{P^\bullet\setminus P}|\nabla\tilde{v}|^2\:dx}{\int_{P^\bullet\setminus P}\tilde{v}^2\:dx}=\overline{R}_{P^\bullet,-\eta}(\tilde{v})\le\max_{S}\overline{R}_{P^\bullet,-\eta}<\overline{\lambda}_{h,-\eta}(P^\bullet)+\varepsilon<0,$$
giving a contradiction.

Let us denote by  $\overline{\alpha_1},\ldots,\overline{\alpha_h}\in\R$ an $h$-tuple of coefficients realizing the maximum of the Rayleigh quotient relative to $P$ in $\text{span}\{u_1|_P,\ldots,u_h|_P\}$.
Notice that, since $P^\bullet\supset P$ and $\partial P\supset\partial P^\bullet$, for any $\varphi\in H^1(P^\bullet)$ it holds
$$
\int_{P^\bullet}|\nabla\varphi|^2\:dx\ge\int_{P}|\nabla\varphi|^2\:dx,\ \int_{P^\bullet}\varphi^2\:dx\ge\int_{P}\varphi^2\:dx,
$$
$$
\int_{\partial P^\bullet}\left[(\varphi^+)^2+(\varphi^-)^2\right]\:d\mathcal{H}^{1}\le\int_{\partial P}\left[(\varphi^+)^2+(\varphi^-)^2\right]\:d\mathcal{H}^{1}.$$
Moreover, since the involved Rayleigh quotients are negative, they are monotonically increasing with respect to the volume integrals and monotonically decreasing with respect to the boundary integral. Taking all these considerations into account we get
\allowdisplaybreaks
\begin{align*}
\overline{\lambda}_{h,-\eta}(P^\bullet)+\varepsilon&\ge\max_{\alpha_1,\ldots,\alpha_h\in\R}\frac{\displaystyle\int_{P^\bullet}\Big|\sum_i\alpha_i\nabla u_i\Big|^2\:dx-\eta\int_{\partial P^\bullet}\left[\Big(\sum_i\alpha_i u_i^-\Big)^2+\Big(\sum_i\alpha_i u_i^+\Big)^2\right]\:d\mathcal{H}^1}{\displaystyle\int_{P^\bullet}\Big(\sum_i\alpha_i u_i\Big)^2\:dx}\\
&\ge\frac{\displaystyle\int_{P^\bullet}\Big|\sum_i\overline{\alpha_i}\nabla u_i\Big|^2\:dx-\eta\int_{\partial P^\bullet}\left[\Big(\sum_i\overline{\alpha_i} u_i^-\Big)^2+\Big(\sum_i\overline{\alpha_i} u_i^+\Big)^2\right]\:d\mathcal{H}^1}{\displaystyle\int_{P^\bullet}\Big(\sum_i\overline{\alpha_i} u_i\Big)^2\:dx}\\
&\ge\frac{\displaystyle\int_{P}\Big|\sum_i\overline{\alpha_i}\nabla u_i\Big|^2\:dx-\eta\int_{\partial P}\left[\Big(\sum_i\overline{\alpha_i} u_i^-\Big)^2+\Big(\sum_i\overline{\alpha_i} u_i^+\Big)^2\right]\:d\mathcal{H}^1}{\displaystyle\int_{P}\Big(\sum_i\overline{\alpha_i} u_i\Big)^2\:dx}\\
&=\max_{\alpha_1,\ldots,\alpha_h\in\R}\frac{\displaystyle\int_{P}\Big|\sum_i\alpha_i\nabla u_i\Big|^2\:dx-\eta\int_{\partial P}\left[\Big(\sum_i\alpha_i u_i^-\Big)^2+\Big(\sum_i\alpha_i u_i^+\Big)^2\right]\:d\mathcal{H}^1}{\displaystyle\int_{P}\Big(\sum_i\alpha_i u_i\Big)^2\:dx}\\
&\ge\overline{\lambda}_{h,-\eta}(P).
\end{align*}
In view of the arbitrariness of $\varepsilon>0$ we get \eqref{eq:bullet} as required.
\end{proof}

\section{Final remarks}\label{sec:Open}
It is interesting to see how the qualitative properties proved in Sections \ref{sec:positive} and \ref{sec:negative} are somehow dual. Indeed, for $\beta>0$, we are able to count the sides of an optimal generalized polygon, but not to show that it is actually a polygon. Unfortunately, removing the possible fractures seems hard in this framework, as we do not have any monotonicity with respect to inclusion, as in the case of Dirichlet boundary conditions. Even with the perimeter constraint we cannot infer anything about the convexity of the minimal polygons, differently to the cases in which a monotonicity holds; see \cite{bucur2004variational} for several examples or \cite{CCL} for a recent application to a fourth order problem.
On the other hand, the case $\beta<0$ allows to remove fractures or holes but not to count the size of the maximal polygon.

\bibliographystyle{abbrv}
\bibliography{bibliography}
\end{document}